\documentclass{amsart}

\usepackage{amsmath, amsthm, amssymb, graphicx, afterpage}

\usepackage[colorlinks]{hyperref}

\newtheorem{theorem}{Theorem}[section]
\newtheorem{lemma}[theorem]{Lemma}

\DeclareMathOperator{\vol}{vol}
\DeclareMathOperator{\tr}{tr}
\DeclareMathOperator{\coeff}{coeff}

\newcommand{\Fcal}{\mathcal{F}}
\newcommand{\Ical}{\mathcal{I}}
\newcommand{\Kcal}{\mathcal{K}}
\newcommand{\Pcal}{\mathcal{P}}
\newcommand{\Scal}{\mathcal{S}}

\newcommand{\so}{{\rm SO}}
\newcommand{\M}{{\rm M}}
\newcommand{\tsp}{{\sf T}}
\newcommand{\fh}{\widehat{f}}
\newcommand{\one}{{\bf 1}}
\newcommand{\Sym}{{\rm S}}

\newcommand{\C}{\mathbb{C}}
\newcommand{\R}{\mathbb{R}}
\newcommand{\Z}{\mathbb{Z}}

\newcommand{\defi}[1]{\textit{#1}}

\newcommand{\dates}[2]{$\mathnormal{#1}$--$\mathnormal{#2}$}
\newcommand{\floor}[1]{\lfloor #1\rfloor}

\title{Computing upper bounds for the packing density of congruent copies of a
  convex body}

\author{Fernando M\'ario de Oliveira Filho}
\address{F.M.~de Oliveira Filho, 
Instituto de Matem\'atica e Estat\'\i stica, 
Universidade de S\~ao Paulo, Rua do Mat\~ao, 1010, 05508-090 S\~ao
Paulo, Brazil}
\email{fmario@gmail.com}

\author{Frank Vallentin} 
\address{F.~Vallentin, Mathematisches Institut, Universit\"at zu
  K\"oln, Weyertal~86--90, 50931 K\"oln, Germany}
\email{frank.vallentin@uni-koeln.de}

\date{July 6, 2016}

\thanks{The first author was support by Rubicon grant 680-50-1014 from
  the Netherlands Organization for Scientific Research (NWO). The
  second author was supported by Vidi grant 639.032.917 from the
  Netherlands Organization for Scientific Research (NWO)}

\subjclass{52C17, 90C22}

\keywords{Tetrahedra packings, pentagon packings, sphere packings,
  Lov\'asz theta number, Delsarte's method, Euclidean motion group,
  polynomial optimization, semidefinite programming}

\begin{document}

\begin{abstract}
  In this paper we prove a theorem that provides an upper bound for
  the density of packings of congruent copies of a given convex body
  in~$\R^n$; this theorem is a generalization of the linear
  programming bound for sphere packings. We illustrate its use by
  computing an upper bound for the maximum density of packings of
  regular pentagons in the plane. Our computational approach is
  numerical and uses a combination of semidefinite programming, sums
  of squares, and the harmonic analysis of the Euclidean motion group.
  We show how, with some extra work, the bounds so obtained can be
  made rigorous.
\end{abstract}

\maketitle

\markboth{F.M. de Oliveira Filho and F. Vallentin}{Computing upper
  bounds for the packing density of congruent copies of a convex
  body}

%
% Introduction
%

\section{Introduction}

How much of Euclidean space can be filled with pairwise nonoverlapping
congruent (i.e., rotated and translated) copies of a given convex
body~$\Kcal$?

A union of congruent copies of~$\Kcal$ with pairwise disjoint
interiors is a \defi{packing} of congruent copies of~$\Kcal$, or just
a packing of~$\Kcal$ for short; below, packings are always packings of
congruent copies of the body.  The \defi{density} of a packing is the
fraction of Euclidean space it covers. Rewritten, the question of the
previous paragraph is: What is the maximum density of a packing of
congruent copies of~$\Kcal$? We call this the \defi{body packing
  problem}.

Theorem~\ref{thm:main}, the main theorem of this paper, provides a way
to compute upper bounds to the density of any packing of a given
convex body~$\Kcal$. We then illustrate the use of this theorem by
applying computational methods to obtain bounds for packings of
regular pentagons in the Euclidean plane --- the case when~$\Kcal
\subseteq \R^2$ is a regular pentagon.

Before presenting our main theorem and its application, let us first
survey some of the most interesting cases of the body packing
problem. We refer to the following books and survey for more
information: Conway and Sloane \cite{ConwayS}, Brass, Moser and Pach
\cite{BrassMP}, Bezdek and Kuperberg \cite{BezdekK}, Fejes T\'oth and
Kuperberg \cite{FejesTothK}.

Perhaps the most well-known case occurs when~$\Kcal$ is a unit
ball. We then have the classical \defi{sphere packing problem}. It is
easy to find out which sphere packings are optimal (i.e., attain the
maximum packing density) in dimensions~$1$ and~$2$. In dimension~$3$,
it was conjectured by the German mathematician and astronomer Johannes
Kepler (\dates{1571}{1630})~\cite{Kepler} that a certain packing
covering~$\pi / (3 \sqrt{2}) = 0.74048\ldots$ of space is optimal. The
Kepler conjecture has been proven by Hales~\cite{Hales}, who makes
massive use of computers in his proof. On August 10, 2014, the
flyspeck project was completed which had the purpose to produce a
formal proof of the Kepler conjecture, see \cite{HalesEtc}.

In all other dimensions the best known upper bounds come from the
linear programming bound of Cohn and Elkies~\cite{CohnE}. For a long
time this upper bound was conjectured to provide tight bounds in
dimensions~$8$ and~$24$ and there was very strong numerical evidence
to support this conjecture, see Cohn and Kumar~\cite{CohnK} or Cohn
and Miller~\cite{CohnM}; the only thing missing was a rigorous
proof. Very recently, on March 14, 2016 Viazovska \cite{Viazovska}
announced a proof for dimension $8$ and a few days later, on March 21,
2016, building on Viazovska's breakthrough result, Cohn, Kumar,
Miller, Radchenko, and Viazovska \cite{CohnKM} announced a proof for
dimension $24$.  In dimensions~$4$, $5$, $6$, $7$, and~$9$ the linear
programming bound of Cohn and Elkies was improved by de Laat,
Oliveira, and Vallentin~\cite{LaatOV}, who also provided upper bounds
for binary sphere packings, that is, for packings of spheres of two
different sizes.

Another case of the body packing problem that has attracted attention
happens when~$\Kcal$ is a regular tetrahedron in~$\R^3$. We called the
sphere packing problem ``classical'', but this adjective most properly
applies to the problem of packing tetrahedra, as it was considered by
Aristotle ($\mathnormal{384}$ BC--$\mathnormal{322}$ BC).

In his treatise \textit{De Caelo} (On the Heavens), Aristotle attacks
the Platonic theory of assigning geometrical figures, namely the
Platonic solids, to the elements, stating (cf.~\textit{De Caelo},
Book~III, Chapter~VIII, in translation by Guthrie~\cite{Guthrie}):
\medbreak

\begin{quote}
This attempt to assign geometrical figures to the simple bodies is on
all counts irrational. In the first place, the whole of space will not
be filled up. Among surfaces it is agreed that there are three figures
which fill the place that contains them --- the triangle, the square,
and the hexagon: among solids only two, the pyramid and the cube. But
they need more than these, since they hold that the elements are more.
\end{quote}
\medbreak

Here, the ``pyramid'' is the regular tetrahedron. Aristotle then
thought it to be possible to tile the space with regular
tetrahedra. Only much latter, Johannes M\"uller von K\"onigsberg
(\dates{1436}{1476}), commonly known as Regiomontanus, a pioneer of
trigonometry, would prove that it is actually not possible to do
so --- it is amusing to observe that this in fact makes Aristotle's
argument stronger.

Regiomontanus' manuscript, titled \textit{De quinque corporibus
  aequilateris quae vulgo regularis nuncupantur: quae videlicet eorum
  locum impleant corporalem et quae non, contra commentatorem
  Aristotelis Averroem}\footnote{On the five equilateral bodies, that
  are usually called regular, and which of them fill their natural
  space, and which do not, in contradiction to the commentator of
  Aristotle, Averro\"es.}, is lost. The Italian mathematician and
astronomer Francesco Maurolico (\dates{1494}{1575}) mentions
Regiomontanus' work on a manuscript of very similar
title~\cite{Maurolico}. Considering Regiomontanus' manuscript as lost,
he sets out to obtain the same results. He observes (cf.~\S2, ibid.)
that the angles between the faces of a solid are of importance in
determining whether the solid tiles space or not. Nowadays one may
easily check that the angle between two faces of a regular tetrahedron
is $\arccos\frac{1}{3} \approx 70.52877^\circ$, thus a little less
than~$360^\circ / 5 = 72^\circ$, and one sees that it is therefore impossible to
tile~$\R^3$ with regular tetrahedra. Maurolico himself did a similar
computation (cf.~\S73, ibid.):
\medbreak

\begin{quote}
(...) Nunc exponam hosce angulos cum suis chordis hic inferius:
\smallskip

\noindent
Pyramidis angulus --
gradus~$70$. minuti\ae~$31$. secund\ae~$43\frac{1}{2}$. chorda $1154701$.\footnote{Below
  I show these angles with their chords:

Angle of the pyramid -- $70$~degrees. $31$~minutes.
$43\frac{1}{2}$~seconds. chord~$1154701$.}
\end{quote}
\medbreak

More on the history of the tetrahedra packing problem can be found in
the paper by Lagarias and Zong~\cite{LagariasZ}. 

In~2006, Conway and Torquato~\cite{ConwayT} found surprisingly dense
packings of tetrahedra. This sparked renewed interest in the problem
and a race for the best construction (cf.~Lagarias and
Zong~\cite{LagariasZ} and Ziegler~\cite{Ziegler}). The current record
is held by Chen, Engel, and Glotzer~\cite{ChenEG}, who found a packing
with density~$\approx 0.8563$, a much larger fraction of space than
that which can be covered by spheres. This prompted the quest for
upper bounds: We know tetrahedra do not tile space, so the maximum
packing density is strictly less than~1. The current record rests with Gravel,
Elser, and Kallus~\cite{GravelEK}, who proved an upper bound of~$1 -
2.6\ldots \cdot 10^{-25}$. They are themselves convinced that the
bound can be greatly improved:
\medbreak

\begin{quote}
  In fact, we conjecture that the optimal packing density corresponds
  to a value of~$\delta$ [the fraction of empty space] many orders of
  magnitude larger than the one presented here. We propose as a
  challenge the task of finding an upper bound with a significantly
  larger value of~$\delta$ (e.g.,~$\delta > 0.01$) and the development
  of practical computational methods for establishing informative
  upper bounds.
\end{quote}
\medbreak

In 1964, in his famous little book on packing and covering Rogers
noted \cite[p.~12]{Rogers}: ``Little precise is known about the
packing density $\delta(\mathcal{K})$ in three or more dimensions.''
Since then the general situation did not improve much. In dimension three, the only cases where the optimal packing density is known are the cases when $\mathcal{K}$ is a space filling polytope, or when $\mathcal{K}$ is the unit ball, or when $\mathcal{K}$ is a slight truncation of the rhombic dodecahedron, see \cite{BezdekK}. In particular finding good upper bounds for the body packing problem is very difficult. Recently, progress was made by Fejes T\'oth, Fodor, and V{\'{\i}}gh \cite{FejesTothFV} in the case when $\mathcal{K}$ is the $n$-dimensional cross polytope.

Our paper can be seen as a step in the search for good upper bounds
for the maximum density of body packings in general and tetrahedra
packings in particular. Its main theorem is a generalization of the
linear programming bound of Cohn and Elkies~\cite{CohnE} for the
sphere packing density, which provides the best known upper bounds in
small and high dimensions (cf.~Cohn and Zhao~\cite{CohnZ}). To specify
a sphere packing it suffices to give the centers of the spheres in the
packing; this is the reason why the Cohn-Elkies bound is a
\textit{linear programming} bound. In our case, to specify a packing
of congruent copies of a body, we need also to consider different
rotations of the body, and so linear programming is replaced by
semidefinite programming.

We apply the theorem to packings of pentagons in the plane because the
specific structure of the Euclidean plane simplifies computations and
because such packings are interesting in themselves (see
e.g.~Kuperberg and Kuperberg~\cite{KuperbergK},
Casselman~\cite{Casselman}, and Atkinson, Jiao, and
Torquato~\cite{AtkinsonJT}), obtaining an upper bound of~$0.98103$ for
the density of any such packing. The best known construction is a
packing consisting of pentagons placed in two opposite orientations
achieving a density of $(5 - \sqrt{5}) / 3 = 0.9213\ldots$
(cf.~Kuperberg and Kuperberg, ibid.). Kallus and Kusner~\cite{KallusK}
showed that the pentagon packing of Kuperberg and Kuperberg cannot be
improved by small perturbations.

Using more refined computational tools, for example using a numerically
stable complex semidefinite programming solver, it is conceivable that
our upper bound could be improved. Our main goal however was to show
how the theorem can be applied and that it gives bounds well-below the
trivial bound of~$1$.  Our long-term goal is to apply the theorem to
obtain upper bounds for packings in~$\R^3$, in particular tetrahedra
packings.

In fact, after the first draft of our paper was finished and
submitted, Hales and Kusner \cite{HalesK} improved our upper bound for
pentagon packings to $0.9611$. For this they used an entirely
different method, based on area estimates of Voronoi cells.

\subsection{The main theorem}

We defined the density of a packing informally, as the fraction of
space it covers. There are different ways to formalize this
definition, and questions arise as to whether every packing has a
density and so on. We postpone such discussion to \S\ref{sec:proof},
when we shall prove the main theorem.

Let~$\so(n)$ be the group of rotations of~$\R^n$, that is,
\[
\so(n) = \{\, A \in \R^{n \times n} : \text{$A^\tsp A = I$ and $\det A
  = 1$}\,\}.
\]
The set~$\M(n) = \R^n \times \so(n)$ is a group with identity
element~$(0, I)$, multiplication defined as
\[
(x, A) (y, B) = (x + Ay, AB),
\]
and inversion given by
\[
(x, A)^{-1} = (-A^{-1} x, A^{-1}).
\]

The group~$\M(n)$, the semidirect product~$\R^n \rtimes \so(n)$, is
the \defi{Euclidean motion group} of~$\R^n$; it is a noncompact (but
locally compact), noncommutative group. When we integrate functions
over~$\M(n)$, we always use the measure~$d(x, A)$, which is the
product of the Lebesgue measure~$dx$ for~$\R^n$ with the Haar
measure~$dA$ for~$\so(n)$, normalized so that~$\so(n)$ has total
measure~$1$.

A bounded complex-valued function~$f\in L^\infty(\M(n))$ is said to be of
\defi{positive type} if 
\[
f(x, A) =
\overline{f((x, A)^{-1})}\quad\text{for all~$(x, A) \in \M(n)$}
\]
and for all~$\rho \in L^1(\M(n))$ we have
\[
\int_{\M(n)} \int_{\M(n)} f((y, B)^{-1} (x, A)) \rho(x, A)
\overline{\rho(y, B)}\, d(y, B) d(x, A) \geq 0.
\]
With this we have all we need for presenting the main theorem.

\begin{theorem}
\label{thm:main}
Let~$\Kcal \subseteq \R^n$ be a convex body and let~$f \in L^1(\M(n))$
be a bounded real-valued function such that:

\begin{enumerate}
\item[(i)] $f$ is continuous and of positive type;

\item[(ii)] $f(x, A) \leq 0$ whenever~$\Kcal^\circ \cap (x + A \Kcal^\circ) =
  \emptyset$, where~$\Kcal^\circ$ is the interior of~$\Kcal$;

\item[(iii)] $\lambda = \int_{\M(n)} f(x, A)\, d(x, A) > 0$.
\end{enumerate}

\noindent
Then the density of any packing of congruent copies of~$\Kcal$ is at
most
\[
\frac{f(0, I)}{\lambda} \vol \Kcal,
\]
where~$\vol \Kcal$ is the volume of~$\Kcal$.
\end{theorem}

This theorem is a generalization of a theorem of Cohn and
Elkies~\cite{CohnE} that provides upper bounds for the maximum density of
sphere packings, and more generally also for translational packings of
convex bodies. The theorem of Cohn and Elkies generalizes 
Delsarte's linear programming method. Delsarte (see for example the
survey of Delsarte and Levenshtein \cite{DelsarteL}) gave a
very general method to determine strong upper bounds for packing
problems in compact spaces. The Cohn-Elkies bound deals with the
noncompact, commutative group $\mathbb{R}^n$ and our
bound deals with the noncompact, noncommutative group $M(n)$.

Our theorem can also be seen as an extension of the Lov\'asz theta
number~\cite{Lovasz}, a parameter originally defined for finite graphs, to the
infinite packing graph for the body~$\Kcal$. Our proof of
Theorem~\ref{thm:main} in \S\ref{sec:proof} relies on this connection
and will make it clear.

Finally, applying Theorem~\ref{thm:main} to find upper bounds for the
densities of packings of a given body~$\Kcal \subseteq \R^n$ means
finding a good function~$f$ satisfying the conditions required in the
theorem. In \S\ref{sec:computations}, to find such a function for the
case of pentagon packings, we use a computational approach that relies
on semidefinite programming (see~\S\ref{sec:sdpsos}) and the harmonic
analysis of the Euclidean motion group of the plane (see
\S\ref{sec:harmonic}.  Here is a place where we see that it is simpler
to deal with pentagon packings than with tetrahedra packings, since
the formulas describing the harmonic analysis of~$\M(2)$ are much
simpler, specially from a computational perspective, than those
describing the harmonic analysis of~$\M(3)$.

\section{Proving the main theorem}
\label{sec:proof}

\subsection{Packing density and periodic packings}

To give a proof of Theorem~\ref{thm:main} we need to present some
technical considerations regarding the density of a packing. Here we
follow Appendix~A of Cohn and Elkies~\cite{CohnE}.

Let~$\Kcal \subseteq \R^n$ be a convex body and~$\Pcal$ be a packing
of congruent copies of~$\Kcal$. We say that the \defi{density}
of~$\Pcal$ is~$\Delta$ if for all~$p \in \R^n$ we have
\[
\Delta = \lim_{r \to \infty} \frac{\vol (B(p, r) \cap \Pcal)}{\vol
  B(p, r)},
\]
where~$B(p, r)$ is the ball of radius~$r$ centered at~$p$.  Not every
packing has a density, but every packing has an \defi{upper density}
given by
\[
\limsup_{r \to \infty} \sup_{p \in \R^n} \frac{\vol (B(p, r) \cap \Pcal)}{\vol
  B(p, r)}.
\]

We say that a packing~$\Pcal$ is \defi{periodic} if there is a
lattice\footnote{A \textit{lattice} is a discrete subgroup of~$(\R^n,
  {+})$.}~$L \subseteq \R^n$ that leaves~$\Pcal$ invariant, i.e.,
which is such that~$\Pcal = x + \Pcal$ for all~$x \in L$; here,~$L$ is
the \defi{periodicity lattice} of~$\Pcal$. In other words, an
invariant packing consists of some congruent copies of~$\Kcal$
arranged inside the fundamental cell of~$L$, and this arrangement
repeats itself at each copy of the fundamental cell translated by
vectors of the lattice.

Periodic packings have well-defined densities. Moreover, it is not
hard to prove that given any packing~$\Pcal$, one may define a
sequence of periodic packings whose fundamental cells have volumes
approaching infinity and whose densities converge to the upper density
of~$\Pcal$. So in computing bounds for the packing density of a given
body, one may restrict oneself to periodic packings. This restriction
is particularly interesting because it allows us to compactify the
problem, as we will see later on.

\subsection{A generalization of the Lov\'asz theta number}
\label{sec:theta}

Let~$G = (V, E)$ be an undirected graph without loops\footnote{Like
  all the graphs considered in this paper.}, finite or
infinite. A set~$I \subseteq V$ is \defi{independent} if no two
vertices in~$I$ are adjacent. The \defi{independence number} of~$G$,
denoted by~$\alpha(G)$, is the maximum cardinality of an independent
set of~$G$.

Packings of a given body~$\Kcal \subseteq \R^n$ correspond to the
independent sets of the \defi{packing graph} of~$\Kcal$. This is the
graph whose vertices are the elements of~$\M(n)$. Here, vertex~$(x, A)
\in \M(n)$ corresponds to the congruent copy~$x + A\Kcal$
of~$\Kcal$. Two vertices are adjacent when the corresponding copies
of~$\Kcal$ cannot both be in the packing at the same time, i.e., when
they intersect in their interiors. In other words, distinct
vertices~$(x, A)$ and~$(y, B)$ are adjacent if
\[
(x + A\Kcal^\circ) \cap (y + B\Kcal^\circ) \neq \emptyset.
\]

Clearly, an independent set of the packing graph corresponds to a
packing and vice versa. The packing graph however has infinite
independent sets, and so its independence number is also infinite.

If we consider periodic packings we may manage to work with graphs
that, though infinite, have a compact vertex set and also a finite
independence number. Given a lattice~$L \subseteq \R^n$, write~$\M(n)
/ L = (\R^n / L) \times \so(n)$. Note that this is a compact
set. Here, we assume that the fundamental cell of~$L$ is big enough so
that there exists a nonempty periodic packing with periodicity
lattice~$L$.

Consider the graph~$G_L$ whose vertex set is~$\M(n) / L$ and in which
two distinct vertices~$(x, A)$ and~$(y, B)$ are adjacent if there
is~$v \in L$ such that
\[
(v + x + A\Kcal^\circ) \cap (y + B\Kcal^\circ) \neq \emptyset.
\]
In this setting, a vertex~$(x, A) \in \M(n) / L$ now represents all
bodies~$v + x + A\Kcal$ for~$v \in L$, and we put an edge between two
distinct vertices if any of the corresponding bodies overlap.

So, independent sets of~$G_L$ correspond to periodic packings with
periodicity lattice~$L$ and vice versa. Moreover,~$G_L$ has finite
independence number, and we actually have that the maximum density of
a periodic packing with periodicity lattice~$L$ is equal to
\begin{equation}
\label{eq:periodic-alpha}
\frac{\alpha(G_L)}{\vol(\R^n / L)} \vol\Kcal.
\end{equation}

In view of the fact that we may restrict ourselves to periodic
packings, as seen in the previous section, if we manage to find an
upper bound for~$\alpha(G_L)$ for every~$L$, then we obtain an upper
bound for the maximum density of any packing of~$\Kcal$.

Computing the independence number of a finite graph is an NP-hard
problem, figuring in the list of combinatorial problems proven to be
NP-hard by Karp~\cite{Karp}. Lov\'asz~\cite{Lovasz} introduced a graph
parameter, the theta number, that provides an upper bound for the
independence number of a finite graph and that can be computed in
polynomial time. In Theorem~\ref{thm:theta} we present a
generalization of the theta number to graphs defined over certain
measure spaces, like the graph~$G_L$.

To present our theorem we need first a few definitions and facts from
functional analysis. For background we refer
the reader to the book by Conway~\cite{Conway}.

Let~$V$ be a separable and compact topological space and~$\mu$ be a
finite Borel measure on~$V$ which is such that every nonempty open
subset of~$V$ has nonzero measure. There are many examples of such a
space. For instance, any finite set~$V$ with the counting measure
provides such an example, as does~$\M(n) / L$ with its natural
measure.

A \defi{Hilbert-Schmidt kernel}, or simply a \defi{kernel}, is a
square-integrable function $K\colon V \times V \to \C$. A kernel
defines an operator~$T_K\colon L^2(V) \to L^2(V)$ as follows: for~$f
\in L^2(V)$ and~$x \in V$ we have
\[
(T_K f)(x) = \int_V K(x, y) f(y)\, d\mu(y).
\]

An \defi{eigenfunction} of~$K$ is a nonzero function~$f \in L^2(V)$ such
that~$T_K f = \lambda f$ for some~$\lambda \in \C$. We say
that~$\lambda$ is the \defi{eigenvalue} associated with~$f$.

A kernel~$K$ is \defi{Hermitian} if $K(x, y) = \overline{K(y, x)}$ for
all~$x$, $y \in V$; a Hermitian kernel defines a self-adjoint
operator~$T_K$. We say that~$K$ is \defi{positive} if it is Hermitian
and for all~$\rho \in L^2(V)$ we have
\[
\int_V \int_V K(x, y) \rho(x) \overline{\rho(y)}\, d\mu(x) d\mu(y)
\geq 0.
\]
This is equivalent to~$\langle T_K \rho, \rho\rangle \geq 0$ for
all~$\rho \in L^2(V)$, where
\[
\langle f, g \rangle = \int_V f(x) \overline{g(x)}\, d\mu(x)
\]
is the standard inner product between $f$, $g \in L^2(V)$. Further
still, a Hermitian kernel is positive if and only if all its
eigenvalues are nonnegative.

\begin{theorem}
\label{thm:theta}
Let~$G = (V, E)$ be a graph, where~$V$ is a separable and compact
topological space having a finite Borel measure~$\mu$ such that every
nonempty open set of~$V$ has nonzero measure. Suppose that
kernel~$K\colon V \times V \to \R$ satisfies the following conditions:

\begin{enumerate}
\item[(i)] $K$ is continuous;

\item[(ii)] $K(x, y) \leq 0$ whenever~$x \neq y$ are nonadjacent;

\item[(iii)] $K - J$ is positive, where~$J$ is the constant~$1$ kernel.
\end{enumerate}

\noindent
Then, for any number~$B$ such that~$B \geq K(x, x)$ for all~$x \in V$,
we have~$\alpha(G) \leq B$.
\end{theorem}

Notice that any kernel~$K$ satisfying the conditions of the theorem
provides an upper bound for~$\alpha(G)$. For a finite graph, the
optimal bound given by the theorem is exactly the theta prime number
of the graph~$G$, a strengthening of the theta number introduced
independently by
McEliece, Rodemich, and Rumsey~\cite{McElieceRR} and
Schrijver~\cite{Schrijver}.

Such generalizations of the theta number have been considered before
by Bachoc, Nebe, Oliveira, and Vallentin~\cite{BachocNOV}, where it is
proved that the linear programming bound of Delsarte, Goethals, and
Seidel~\cite{DelsarteGS} for the sizes of spherical codes comes from
such a generalization.

To prove the theorem we need the following alternative
characterization of continuous and positive kernels: A continuous
kernel~$K\colon V \times V \to \C$ is positive if and only if for
all~$m$ and any
choice~$x_1$, \dots,~$x_m$ of points in~$V$ the matrix $\bigl(K(x_i,
x_j)\bigr)_{i,j=1}^m$ is positive semidefinite (cf.~Lemma~1 in
Bochner~\cite{Bochner}). In fact, here is where the hypothesis on~$V$
is used: To prove this, one needs to use the fact that~$V$ is compact,
separable, and that the measure~$\mu$ is finite and nonzero on
nonempty open sets.

\begin{proof}[Proof of Theorem~\ref{thm:theta}]
Let~$I \subseteq V$ be a nonempty finite independent set. Since~$K - J$ is a
continuous and positive kernel, we have that
\[
\sum_{x, y \in I} K(x, y) \geq \sum_{x, y \in I} J(x, y) = |I|^2.
\]
Since~$K$ satisfies condition~(ii), if~$B$ is an upper bound on
the diagonal entries of~$K$ we have that
\[
|I| B \geq \sum_{x \in I} K(x, x) \geq \sum_{x, y \in I} K(x, y) \geq |I|^2,
\]
and then~$B \geq |I|$, as we wanted.
\end{proof}

\subsection{A proof of the main theorem}
\label{sec:main-proof}

The proof of Theorem~\ref{thm:main} is similar to the proof of
Theorem~3.1 in the paper by de Laat, Oliveira, and
Vallentin~\cite{LaatOV}.  We first prove the theorem for functions of
bounded support and then extend it to~$L^1$ functions.

Let~$f\colon \M(n) \to \R$ be a function of bounded support
satisfying conditions~(i)--(iii) in Theorem~\ref{thm:main}. Given a
lattice~$L \subseteq \R^n$ whose fundamental cell is big enough so
that there is a nonempty periodic packing with periodicity
lattice~$L$, we use~$f$ to define a kernel $K_L\colon (\M(n) / L)
\times (\M(n) / L) \to \R$ satisfying conditions~(i)--(iii) of
Theorem~\ref{thm:theta} for the graph~$G_L$, defined in the previous
section.

In fact, we let
\begin{equation}
\label{eq:kernel-def}
\begin{split}
K_L((x, A), (y, B))&= \sum_{v \in L} f((y - v, B)^{-1} (x, A))\\
&=\sum_{v \in L} f(B^{-1}(x - y + v), B^{-1}A)
\end{split}
\end{equation}
for every~$(x, A)$, $(y, B) \in \M(n) / L$.

Since~$f$ has bounded support and~$x$, $y \in \R^n / L$, the sum above
is actually a finite sum. This shows not only that~$K_L$ is
well-defined, but also that it is continuous.

We claim that~$K_L$ has the following properties:

\begin{enumerate}
\item[K1.] it is a positive kernel;

\item[K2.] the constant~$1$ function is an eigenfunction of~$K_L$, with
  eigenvalue~$\lambda$;

\item[K3.] $K_L((x, A), (y, B)) \leq 0$ if~$(x, A) \neq (y, B)$ are
  nonadjacent in~$G_L$;

\item[K4.] $f(0, I) \geq K_L((x, A), (x, A))$ for all~$(x, A) \in \M(n) /
  L$.
\end{enumerate}

Once we have established these properties, it becomes clear that the
kernel
\[
\tilde{K}_L = \frac{\vol(\R^n / L)}{\lambda} K_L
\]
satisfies conditions (i)--(iii) of Theorem~\ref{thm:theta} for the
graph~$G_L$. In particular, the fact that~$\tilde{K}_L - J$ is
positive follows directly from~K1 and~K2 above, because the
constant~$1$ function is an eigenfunction of both~$\tilde{K}_L$
and~$J$ with associated eigenvalue~$\vol(\R^n / L)$ in both cases.

 But then from~K4 we may
take~$B = f(0, I) \vol(\R^n / L) / \lambda$ in Theorem~\ref{thm:theta}
and obtain the bound
\[
\alpha(G_L) \leq f(0, I) \frac{\vol(\R^n / L)}{\lambda}.
\]
So the maximum density of any periodic packing with periodicity
lattice~$L$ is bounded from above by (cf.~equation~\eqref{eq:periodic-alpha})
\[
\frac{\alpha(G_L)}{\vol(\R^n / L)} \vol\Kcal \leq \frac{f(0,
  I)}{\lambda} \vol\Kcal,
\]
and since~$L$ is an arbitrary lattice, we would have a proof of
Theorem~\ref{thm:main}.

So we set out to prove~K1--K4. Property~K1 is implied by the fact
that~$f$ is of positive type. In fact, since~$f(x, A) = \overline{f((x,
A)^{-1})}$, kernel~$K_L$ is Hermitian by construction. Now take a
function~$\rho \in L^2(\M(n) / L)$. We also view~$\rho$ as the
periodic function~$\rho\colon \M(n) \to \C$ such that~$\rho(x + v, A)
= \rho(x, A)$ for all~$v \in L$. For~$T > 0$, write~$\M_T(n) = [-T,
T]^n \times \so(n)$. Then
\[
\begin{split}
&\int_{\M(n)/L} \int_{\M(n)/L} K_L((x, A), (y, B)) \rho(x, A)
  \overline{\rho(y, B)}\, d(y, B) d(x, A)\\
&\qquad = \int_{\M(n) / L} \int_{\M(n) / L} \sum_{v \in L} f((y - v,
  B)^{-1} (x, A)) \rho(x, A) \overline{\rho(y, B)}\, d(y, B) d(x, A).\\
&\qquad =\int_{\M(n) / L} \int_{\M(n)} f((y, B)^{-1} (x, A))
  \overline{\rho(y, B)}\, d(y, B) \rho(x, A)\, d(x, A)\\
&\qquad = \lim_{T \to \infty} \frac{\vol(\R^n / L)}{\vol [-T, T]^n}
  \int_{\M_T(n)} \int_{\M(n)} f((y, B)^{-1} (x, A)) \overline{\rho(y,
    B)}\, d(y, B)\\
&\hskip10cm{}\cdot \rho(x, A)\, d(x, A)\\
&\qquad = \lim_{T \to \infty} \frac{\vol(\R^n / L)}{\vol [-T, T]^n}
  \int_{\M_T(n)} \int_{\M_T(n)} f((y, B)^{-1} (x, A)) \rho(x, A)\overline{\rho(y,
    B)}\\
&\hskip10cm{}\cdot d(y, B) d(x, A)\\
&\qquad\geq 0.
\end{split}
\]

Above, from the second to the third line we exchange the sum with the
innermost integral and use the fact that, if~$h\colon \R^n \to \C$ is
an integrable function, then
\begin{equation}
\label{eq:periodic-sum}
\sum_{v \in L} \int_{\R^n / L} h(x + v)\, dx = \int_{\R^n} h(x)\, dx.
\end{equation}

To go from the third to the fourth line we notice that the function
\[
(x, A) \mapsto \int_{\M(n)} f((y, B)^{-1} (x, A)) \overline{\rho(y,
  B)}\, d(y, B)
\]
is periodic with respect to the lattice~$L$. From the fourth to the
fifth line we use the fact that~$f$ is of bounded support. Finally,
from the fifth to the sixth line we apply directly the definition of a
function of positive type.

% TODO: use condition (iii) of the theorem.
To see~K2, we use~\eqref{eq:periodic-sum} and notice that for a
fixed~$(x, A) \in \M(n) / L$ we have
\[
\begin{split}
\int_{\M(n) / L} K_L((x, A), (y, B))\, d(y, B)
&= \int_{\M(n) / L} \sum_{v \in L} f((y - v, B)^{-1} (x, A))\, d(y,
B)\\
&= \int_{\M(n)} f((y, B)^{-1} (x, A))\, d(y, B)\\
&= \lambda.
\end{split}
\]

To prove~K3, recall that~$(x, A)$, $(y, B) \in \M(n) / L$ are
nonadjacent if for all~$v \in L$ we have~$(v + x + A\Kcal^\circ) \cap
(y + B\Kcal^\circ) = \emptyset$, and this is the case if and only
if
\[
\Kcal^\circ \cap (B^{-1}(x - y + v) + B^{-1} A \Kcal^\circ) =
\emptyset.
\]
But then, since~$f$ satisfies~(ii) in the statement of
Theorem~\ref{thm:main}, every summand in~\eqref{eq:kernel-def} will be
nonpositive, implying~K3.

Property K4 may be similarly proven. In fact, since from start
we assumed~$L$ has a large enough fundamental cell, for~$v \in L$
with~$v \neq 0$ we have~$\Kcal^\circ \cap (A^{-1} v +
\Kcal^\circ) = \emptyset$. But then in
expression~\eqref{eq:kernel-def} for~$K_L((x, A), (x, A))$, all
summands but the one for~$v = 0$ will be nonpositive, and the summand
for~$v = 0$ is exactly~$f(0, I)$, proving~K4.

So we have K1--K4, and Theorem~\ref{thm:main} follows for
functions~$f$ of bounded support. To prove the theorem for a
given~$L^1$ function, we approximate it by functions of bounded
support as follows.

Let~$f \in L^1(\M(n))$ be a real-valued function
satisfying conditions (i)--(iii)  in Theorem~\ref{thm:main}. For~$T >
0$, consider the function~$g_T\colon \M(n) \to \R$ given by
\[
g_T(x, A) = \frac{\vol(B(0, T) \cap B(x, T))}{\vol B(0, T)} f(x, A),
\]
where~$B(x, T)$ is the ball of radius~$T$ centered at~$x$.

Clearly, $g_T$ is continuous and has bounded support. We claim that it
is also a function of positive type.

To see this, we will use a characterization of continuous functions of
positive type analogous to the characterization of continuous
positive kernels given in \S\ref{sec:theta}, namely: A continuous
function~$f \in L^\infty(\M(n))$ is of positive type if and only if
the matrix
\[
\bigl(f((x_j, A_j)^{-1} (x_i, A_i))\bigr)_{i,j=1}^m
\]
is positive semidefinite for any~$m$ and any elements~$(x_1, A_1)$,
\dots,~$(x_m, A_m) \in \M(n)$ (cf.~Folland~\cite{Folland},
Proposition~3.35).

Let~$(x_1, A_1)$, \dots,~$(x_m, A_m) \in \M(n)$ be any given
elements. Let~$\chi_i\colon \R^n \to \{0, 1\}$ be the characteristic
function of~$B(x_i, T)$ and denote by~$\langle f, g \rangle$ the
standard inner product between functions~$f$, $g \in L^2(\R^n)$. Then
\[
\begin{split}
g_T((x_j, A_j)^{-1} (x_i, A_i)) 
&= g_T(A_j^{-1} (x_i - x_j), A_j^{-1} A_i)\\
&= \frac{\vol(B(0, T) \cap B(A_j^{-1} (x_i - x_j), T))}{\vol B(0, T)}
   f((x_j, A_j)^{-1} (x_i, A_i))\\
&=\frac{\vol(B(x_i, T) \cap B(x_j, T))}{\vol B(0, T)}
   f((x_j, A_j)^{-1} (x_i, A_i))\\
&=\frac{\langle \chi_i, \chi_j\rangle}{\vol B(0, T)} f((x_j, A_j)^{-1}
   (x_i, A_i)).
\end{split}
\]

This shows that the matrix
\begin{equation}
\label{eq:gt-matrix}
\bigl(g_T((x_j, A_j)^{-1} (x_i, A_i))\bigr)_{i,j=1}^m
\end{equation}
is the Hadamard (entrywise) product of the matrices
\[
\bigl(f((x_j, A_j)^{-1} (x_i, A_i))\bigr)_{i,j=1}^m\qquad
\text{and}\qquad
\frac{1}{\vol B(0, T)} \bigl(\langle \chi_i,
\chi_j\rangle\bigr)_{i,j=1}^m.
\]
The first matrix above is positive semidefinite since~$f$ is of
positive type. The second matrix is positive semidefinite since it is
a positive multiple of the Gram matrix of vectors~$\chi_1$,
\dots,~$\chi_m$. So we have that~\eqref{eq:gt-matrix} is positive
semidefinite, and thus~$g_T$ is of positive type.

By construction, whenever~$f(x, A) \leq 0$, also~$g_T(x, A) \leq
0$. So~$g_T$ is a continuous function of bounded support satisfying
conditions~(i) and~(ii) from the statement of
Theorem~\ref{thm:main}. This implies immediately that
\begin{equation}
\label{eq:gt-bound}
\frac{g_T(0, I)}{\lambda_T} \vol\Kcal = \frac{f(0, I)}{\lambda_T}
\vol\Kcal
\end{equation}
is an upper bound for the density of any packing of congruent copies
of~$\Kcal$, where
\[
\lambda_T = \int_{\M(n)} g_T(x, A)\, d(x, A).
\]

To finish, notice that~$g_T$ converges pointwise to~$f$ as~$T \to
\infty$. Moreover, for all~$T$ we have~$|g_T(x, A)| \leq |f(x, A)|$. So
it follows from Lebesgue's dominated convergence theorem
that~$\lambda_T \to \lambda$ as~$T \to \infty$. This together
with~\eqref{eq:gt-bound} finishes the proof of Theorem~\ref{thm:main}.

\subsection{Using the symmetry of the body}

Let~$\Kcal \subseteq \R^n$ be a convex body. Its \defi{symmetry group}
is the subgroup of~$\so(n)$ defined as
\[
\Sym(\Kcal) = \{\, A \in \so(n) : A\Kcal = \Kcal\,\}.
\]

The action by conjugation of an element~$B \in \so(n)$ on a function~$f
\in L^1(\M(n))$ is given by
\[
(B \cdot f)(x, A) = f((0, B) (x, A) (0, B)^{-1}).
\]

Suppose now~$G$ is a compact subgroup of~$\Sym(\Kcal)$. Then in
Theorem~\ref{thm:main} we may restrict ourselves to $G$-invariant
functions~$f \in L^1(\M(n))$ without affecting the bound
obtained. Here we say that~$f$ is \defi{$G$-invariant} if~$B \cdot f =
f$ for all~$B \in G$.

This restriction to $G$-invariant functions may make it easier to
apply Theorem~\ref{thm:main}. This is actually the case for our
application to pentagon packings, as we will see in~\S\ref{sec:computations}.

To see that the restriction to $G$-invariant functions does not affect
the bound that can be obtained from Theorem~\ref{thm:main}, notice
that, if~$f \in L^1(\M(n))$ is a bounded continuous function satisfying
conditions (i)--(iii) of Theorem~\ref{thm:main}, then also~$B\cdot f$,
for~$B \in G$, satisfies these conditions.

In fact, to show that~$B \cdot f$ is of positive type, let~$(x_1,
A_1)$, \dots,~$(x_m, A_m) \in \M(n)$. Then
\[
\begin{split}
& \bigl((B \cdot f)((x_j, A_j)^{-1} (x_i, A_i))\bigr)_{i,j=1}^m\\
= \; & \bigl(f((0, B) (x_j, A_j)^{-1} (x_i, A_i) (0,
B)^{-1})\bigr)_{i,j=1}^m\\
= \; & \bigl(f(((x_j, A_j)(0, B)^{-1})^{-1} ((x_i, A_i)(0, B)^{-1}))\bigr)_{i,j=1}^m,
\end{split}
\]
and since~$f$ is of positive type,~$B \cdot f$ is also of positive
type (cf.\ the alternative characterization of continuous functions of
positive type in the previous section).

To see that~$B\cdot f$ satisfies condition~(ii) of
Theorem~\ref{thm:main}, notice that, since~$B^{-1} \Kcal =
\Kcal$, we have~$\Kcal^\circ \cap (x + A \Kcal^\circ) = \emptyset$ if
and only if $\Kcal^\circ \cap (Bx + B A B^{-1} \Kcal^\circ) =
\emptyset$.

Finally, we have
\[
\int_{\M(n)} (B \cdot f)(x, A)\, d(x, A) =
\int_{\M(n)} f(x, A)\, d(x, A),
\]
and so we see that~$B\cdot f$ satisfies the conditions of
Theorem~\ref{thm:main} and provides the same bound as~$f$. 

Now, since~$G$ is compact, it admits a Haar measure~$\mu$ which we
normalize so that~$\mu(G) = 1$. Then it is immediate that the
function~$\overline{f} \in L^1(\M(n))$ such that
\[
\overline{f}(x, A) = \int_G (B \cdot f)(x, A)\, d\mu(B)
\]
satisfies (i)--(iii) of Theorem~\ref{thm:main} and provides the same
bound as~$f$. Moreover,~$\overline{f}$ is $G$-invariant. So it follows
that a restriction to $G$-invariant functions does not affect the
bound of Theorem~\ref{thm:main}.

\section{Semidefinite programming and sums of squares}
\label{sec:sdpsos}

We collect here the basic facts we need from semidefinite
programming. For further background we refer to the book by Ben-Tal
and Nemirovski~\cite{BentalN}.

A linear programming problem amounts to maximizing a linear
function over a polyhedron, which is the intersection of the nonnegative
orthant $\R^n_{\geq 0}$ with an affine subspace. A semidefinite
programming problem --- a rich generalization of linear
programming --- amounts to maximizing a linear function over a
spectrahedron, the intersection of the cone of positive semidefinite
matrices $\mathcal{S}^n_{\succeq 0}$ with an affine subspace.  A
semidefinite programming problem in primal standard form is
\[
\sup\left\{\,\langle C, X \rangle : X \in \mathcal{S}^n_{\succeq 0},\;
  \langle A_j, X \rangle = b_j,\; j = 1, \ldots, m\,\right\},
\]
where $C$, $A_1$, \dots,~$A_m$ are given $n
\times n$ matrices and where $b_1$, \dots,~$b_m \in \C$. Here
$\langle A, B \rangle = \tr(B^* A)$ denotes the trace inner product
between matrices. Matrices~$C$ and~$A_i$ are usually required to be
symmetric (or Hermitian). The seemingly more general setting used here
can be easily reduced to this restricted version though.

Semidefinite programming problems are conic optimization
problems. Sometimes it is convenient to assume that the variable
matrix~$X$ has block-diagonal structure, which amounts to changing the
cone~$\mathcal{S}^n_{\succeq 0}$ to the direct product
$\mathcal{S}^{n_1}_{\succeq 0} \times \cdots \times
\mathcal{S}^{n_k}_{\succeq 0}$.
For solving semidefinite programming problems two types of algorithms
are available: the ellipsoid method and interior point methods. The
ellipsoid method focuses on the existence of polynomial-time
algorithms but no practical implementation is available. In contrast
to this there are many very good implementations of interior point
methods; De Klerk and Vallentin showed in \cite{KlerkV} that a
variant of the interior point method for semidefinite programming
can run in polynomial time on the Turing machine model.

Semidefinite programming is specially useful for certifying
the nonnegativity of polynomials or of trigonometric polynomials via
sums of squares. We quickly discuss the univariate case
--- the multivariate case is a simple extension.

A univariate polynomial $p \in \mathbb{R}[x]$
of degree $2d$ is a sum of squares,
i.e., it can be written as
\[
p = h_1^2 + \cdots + h_r^2 \quad \text{for some $r \in \mathbb{N}$ and
  $h_1, \ldots, h_r \in \R[x]$ of degree at most $d$}
\]
if and only if there is a positive semidefinite matrix $Q$ with
\[
p = \langle V, Q \rangle,
\]
where~$V$ is a matrix of polynomials such that~$V_{kl} = P_k(x)
P_l(x)$ for some basis~$P_k$ of the space of polynomials of degree at
most~$d$.

Note $p = \langle V, Q \rangle$ is an
identity between polynomials. One can check it by linear
equalities --- equating the coefficients --- once one writes both sides
in terms of some basis.

If a polynomial can be written as a sum of squares, then it is clearly
nonnegative. For univariate polynomials, the converse is also
true. This is not the case in general, however; Laurent~\cite{Laurent}
presents a survey.

A similar approach can be applied to trigonometric polynomials. Such
is an expression of the sort
\[
p(\theta) = \sum_{k=-n}^n c_k e^{ik\theta},
\]
where~$c_k = \overline{c_{-k}}$. One way to certify that this
trigonometric polynomial is nonnegative for all~$\theta$ is to write
it as a sum of squares, that is, to write it as
\[
p(\theta) = |h_1(e^{i\theta})|^2 + \cdots + |h_r(e^{i\theta})|^2
\]
for some number~$r$ and some univariate polynomials~$h_1$,
\dots,~$h_r$. Now, being a sum of squares is equivalent to the
existence of an~$(n+1) \times (n+1)$ positive semidefinite matrix~$Q$
such that
\[
p(\theta) = \langle V(e^{i\theta}), Q \rangle,
\]
where~$V$ is the matrix with~$V_{kl}(z) = z^{k-l}$.

\section{Harmonic analysis on $\M(2)$}
\label{sec:harmonic}

Our approach to apply Theorem~\ref{thm:main} in order to obtain upper
bounds for the maximum density of pentagon packings is to specify the
function~$f$ via its Fourier transform. So here we quickly present
the facts from the theory of harmonic analysis on~$\M(2)$ that we will
use. We follow Sugiura~\cite{Sugiura} closely, though we deviate at
some points, mainly concerning choices of normalization, as we will
see.

For $x$, $y \in \R^2$, denote by~$x \cdot y = x_1 y_1 + x_2 y_2$ the
Euclidean inner product. Let~$S^1$ be the unit circle and for~$\varphi$,
$\psi \in L^2(S^1)$ denote by~$\langle \varphi, \psi\rangle$ the
standard inner product between~$\varphi$ and~$\psi$, i.e.,
\[
\langle\varphi, \psi\rangle = \frac{1}{\omega(S^1)} \int_{S^1}
\varphi(\xi) \overline{\psi(\xi)}\, d\omega(\xi),
\]
where~$\omega$ is the Lebesgue measure on the unit circle.

For~$a \geq 0$ and~$(x, A) \in \M(2)$, consider the operator~$U^a_{(x,
  A)}\colon L^2(S^1) \to L^2(S^1)$ defined as follows: For~$\varphi
\in L^2(S^1)$ we have
\[
[U^a_{(x, A)}\varphi](\xi) = e^{2\pi i a x \cdot \xi} \varphi(A^{-1}
\xi)
\]
for all~$\xi \in S^1$. (In the definition of~$U^a_{(x, A)}$ we differ
from Sugiura~\cite{Sugiura}, who omits the factor~$2\pi$, which we
include to obtain better formulas --- from a computational point of
view --- later on. Besides changing some normalization parameters, this
does not affect the theory as presented by Sugiura.)

Operator~$U^a_{(x, A)}$ is a bounded and unitary operator. Moreover,
one can easily check that 
\[
U^a_{(x, A)(y, B)} = U^a_{(x, A)} U^a_{(y, B)}
\]
for all~$(x, A)$, $(y, B) \in \M(2)$. So the strongly continuous
map~$\rho_a(x, A) = U^a_{(x, A)}$ provides a representation of~$\M(2)$
for every~$a \geq 0$. Representations~$\rho_a$, for~$a > 0$, are
all irreducible and pairwise nonequivalent.

Given a function~$f \in L^1(\M(2))$, its Fourier transform
at~$a \geq 0$ is the bounded
operator~$\widehat{f}(a)\colon L^2(S^1) \to L^2(S^1)$ defined as
\[
\widehat{f}(a) = \int_{\M(2)} f(x, A) U^a_{(x, A)^{-1}}\, d(x, A).
\]

Having defined the Fourier transform of~$f$, we would like to have
an \textit{inversion formula}, that is, a way to compute~$f$ back from
its Fourier transform. In our case the inversion formula takes the
following shape:
\begin{equation}
\label{eq:inv-formula}
f(x, A) = 2\pi \int_0^\infty \tr(U^a_{(x, A)} \widehat{f}(a)) a\, da,
\end{equation}
where $\tr F$ is the trace of a trace-class operator~$F$.  In the
following, we only need positive trace-class operators. We define them
now briefly, and we refer e.g.\ to Conway~\cite{Conway} or
Folland~\cite{Folland} for further information. A positive bounded
operator $F \colon L^2(S^1) \to L^2(S^1)$ is called
\textit{trace-class} if there is a complete orthonormal system
$\varphi_k$ consisting of eigenfunctions of $F$ with eigenvalues
$\lambda_k \geq 0$ and $\sum_k \lambda_k < \infty$. Then the trace of
$F$ is
$\tr(F) = \sum_k \lambda_k = \sum_k \langle F\varphi_k, \varphi_k
\rangle$.

In \eqref{eq:inv-formula} we again deviate slightly from the
exposition of Sugiura. The extra factor of~$2\pi$ in the above formula
as compared to Theorem~3.1 in his book~\cite{Sugiura} follows from the
different normalization he uses for the measure of~$\M(2)$.

Of course, it is not always the case that the inversion formula holds
or converges everywhere. In the book by Sugiura it is shown that the
inversion formula holds for \textit{rapidly decreasing functions} (see
Definition~3 in Chapter~IV, \S3 in Sugiura~\cite{Sugiura}).

We will provide explicit formulas for the Fourier transform and hence
obtain explicit formulas for~$f$. In this way it will be clear in our
application that~$f$ is continuous and~$L^1$. To ensure that~$f$ is of
positive type, the following lemma will be useful. It shows that one
can parametrize positive type functions by the positivity of the
Fourier transform $\widehat{f}$. Here it shows that the computations
needed for applying Theorem~\ref{thm:main} are much more complicated
than those for applying the Cohn-Elkies bound. For the Euclidean
motion group $M(n)$ the Fourier transform of a positive type function
is positive trace-class-operator-valued, whereas for the
translation group $\R^n$ its values are simply nonnegative real
numbers.

\begin{lemma}
\label{lem:pos-type}
Suppose that for each~$a \geq 0$ we have that~$\widehat{f}(a)$ is a
positive, trace-class operator. Then, if the
function~$f$ defined in~\eqref{eq:inv-formula} is bounded and
continuous, it is of positive type.
\end{lemma}

\begin{proof}
Take~$(x_1, A_1)$, \dots,~$(x_m, A_m) \in \M(2)$. Recalling the
alternative characterization of continuous functions of positive type
given in \S\ref{sec:main-proof}, we show that the matrix
\[
\bigl(f((x_j, A_j)^{-1} (x_i, A_i))\bigr)_{i,j=1}^m
\]
is positive semidefinite.

By construction, this is a Hermitian
matrix. From~\eqref{eq:inv-formula}, to prove it is positive
semidefinite it suffices to show that for all~$a \geq 0$ the matrix
\begin{equation}
\label{eq:mat-inv}
\bigl(\tr(U^a_{(x_j, A_j)^{-1} (x_i, A_i)}
\widehat{f}(a))\bigr)_{i,j=1}^m
\end{equation}
is positive semidefinite.

Notice that since each~$\widehat{f}(a)$ is trace-class, and
since~$U^a_{(x, A)}$ is a bounded operator, then~$U^a_{(x, A)}
\widehat{f}(a)$ is trace-class for all~$(x, A) \in \M(2)$, and so each
entry of~\eqref{eq:mat-inv} is well-defined.

To see that~\eqref{eq:mat-inv} is positive semidefinite,
let~$\varphi_1$, $\varphi_2$, \dots\ be a complete orthonormal system
of~$L^2(S^1)$. For~$i$, $j = 1$, \dots,~$m$ we have
\[
\begin{split}
\tr(U^a_{(x_j, A_j)^{-1} (x_i, A_i)} \widehat{f}(a))
&= \sum_{k=1}^\infty \langle U^a_{(x_j, A_j)^{-1} (x_i, A_i)}
\widehat{f}(a) \varphi_k, \varphi_k \rangle\\
&= \sum_{k=1}^\infty \langle U^a_{(x_j, A_j)^{-1} (x_i, A_i)}
\widehat{f}(a) U^a_{(x_j, A_j)^{-1}}\varphi_k, U^a_{(x_j,
  A_j)^{-1}}\varphi_k \rangle\\
&= \sum_{k=1}^\infty \langle U^a_{(x_i, A_i)}
\widehat{f}(a) U^a_{(x_j, A_j)^{-1}}\varphi_k, \varphi_k \rangle\\
&= \sum_{k=1}^\infty \langle \widehat{f}(a) U^a_{(x_j,
    A_j)^{-1}}\varphi_k, U^a_{(x_i, A_i)^{-1}}\varphi_k \rangle\\
&= \sum_{k=1}^\infty \langle \widehat{f}^{1/2}(a) U^a_{(x_j,
    A_j)^{-1}}\varphi_k, \widehat{f}^{1/2}(a) U^a_{(x_i, A_i)^{-1}}\varphi_k \rangle.
\end{split}
\]
Here, we go from the first to the second line by noticing that,
since~$U^a_{(x_j, A_j)^{-1}}$ is a unitary operator, $U^a_{(x_j,
  A_j)^{-1}} \varphi_1$, $U^a_{(x_j, A_j)^{-1}} \varphi_2$, \dots\ is
also a complete orthonormal system of~$L^2(S^1)$. Finally, from the
fourth to the fifth line, we observe that since~$\widehat{f}(a)$ is a
positive trace-class operator, it has a
square-root~$\widehat{f}^{1/2}(a)$, a self-adjoint operator such
that~$\widehat{f}(a) = \widehat{f}^{1/2}(a) \widehat{f}^{1/2}(a)$.

So we see that~\eqref{eq:mat-inv} is a sum of positive semidefinite
matrices, the $k$th summand being the Gram matrix of~$m$ vectors, and
we are done.
\end{proof}

We finish this section by computing a more explicit formula for the
inverse transform. We identify both~$\so(2)$ and~$S^1$ with the
torus~$\R / (2\pi\Z)$, and by an abuse of language with the
interval~$[0, 2\pi]$. We equip~$L^2([0, 2\pi])$ with the inner product
\[
\langle \varphi, \psi \rangle = \frac{1}{2\pi} \int_0^{2\pi}
\varphi(\xi) \overline{\psi(\xi)}\, d\xi
\]
for $\varphi$, $\psi \in L^2([0, 2\pi])$. Then the functions~$\chi_r
\in L^2([0, 2\pi])$, for~$r \in \Z$, defined as~$\chi_r(\xi) = e^{i r
  \xi}$ provide a complete orthonormal system of~$L^2([0, 2\pi])$.

We define the \defi{matrix coefficients} of the operator~$U^a_{(x, A)}$ on the
basis~$\chi_r$ as
\[
u^a_{r,s}(x, A) = \langle U^a_{(x, A)} \chi_s,
\chi_r\rangle\quad\text{with~$r$, $s \in \Z$}.
\]
To compute this, we express~$x$ in polar coordinates as
\[x = \rho
(\cos\theta, \sin\theta)
\]
and we see~$A$ as the rotation matrix
\begin{equation}
\label{eq:A-alpha}
A = A(\alpha) = \begin{pmatrix}
\cos\alpha&-\sin\alpha\\
\sin\alpha&\cos\alpha
\end{pmatrix},
\end{equation}
which rotates vectors counter-clockwise by an angle of~$\alpha$. Then
\begin{equation}
\label{eq:u-coeffs}
\begin{split}
u^a_{r,s}(\rho, \theta, \alpha)
&= \langle U^a_{(x, A)} \chi_s, \chi_r\rangle \\
&= \frac{1}{2\pi} \int_0^{2\pi} [U^a_{(x, A)} \chi_s](\xi)
\overline{\chi_r(\xi)}\, d\xi\\
&=\frac{1}{2\pi} \int_0^{2\pi} e^{2\pi i a \rho (\cos \theta,
  \sin\theta) \cdot (\cos \xi, \sin \xi)} e^{is(\xi-\alpha)}
e^{-ir\xi}\, d\xi\\
&=\frac{1}{2\pi} e^{-i s \alpha} \int_0^{2\pi} e^{2\pi i a \rho
  \cos(\xi - \theta)} e^{i(s - r)\xi}\, d\xi\\
&=\frac{1}{2\pi} e^{-is\alpha} \int_0^{2\pi} e^{2\pi i a \rho\cos\xi}
e^{i(s - r)(\xi + \theta)}\, d\xi\\
&=\frac{1}{2\pi} e^{-i(s\alpha + (r - s)\theta)} \int_0^{2\pi}
e^{i(s-r)\xi} e^{2\pi i a \rho\cos \xi}\, d\xi\\
&=i^{s - r} e^{-i(s\alpha + (r - s)\theta)} J_{s - r}(2\pi a\rho).
\end{split}
\end{equation}
Here,~$J_n(z)$ is the Bessel function of parameter~$n$. To obtain the
last line, we apply Bessel's integral (cf.~Watson~\cite{Watson},
(1)~in Chapter~II, \S2.2).

We may then rewrite~\eqref{eq:inv-formula} by expressing the
operators~$\widehat{f}(a)$ on the basis~$\chi_r$, for~$r \in \Z$. This
gives us
\begin{equation}
\label{eq:f-formula}
\begin{split}
f(\rho, \theta, \alpha)& = \int_0^\infty \sum_{r,s\in \Z}
\widehat{f}(a)_{r,s} u^a_{r,s}(\rho, \theta, \alpha) a\, da\\
&= \int_0^\infty \sum_{r,s\in \Z}
\widehat{f}(a)_{r,s} i^{s - r} e^{-i(s\alpha + (r - s)\theta)} J_{s -
  r}(2\pi a\rho) a\, da.
\end{split}
\end{equation}

\section{Computations for pentagon packings}
\label{sec:computations}

In this section we present a semidefinite programming problem and show
how from its solution a function~$f$ can be derived that satisfies
conditions (i)--(iii) of Theorem~\ref{thm:main} when~$\Kcal$ is a
regular pentagon. We describe the semidefinite programming problem in
detail, and then discuss how it can be solved with the computer and
how a function~$f$ can be obtained from its solution that provides the
bound of~$0.98103$ for the maximum density of packings of regular
pentagons in~$\R^2$.

Throughout this section, $\Kcal$ will denote the regular pentagon
on~$\R^2$ whose vertices are the points
\[
\frac{1}{2} (\cos (k 2\pi / 5), \sin(k 2\pi /5))\quad
\text{for~$k = 0$, \dots,~$4$}.
\]
Note the circumscribed circle of~$\Kcal$ has radius~$1/2$.

The symmetry group of~$\Kcal$ is isomorphic to~$C_5$, the cyclic group
of order~$5$. It consists of the rotation matrices~$A(k 2\pi / 5)$,
for~$k = 0$, \dots,~$4$, where~$A(\alpha)$ is
given in~\eqref{eq:A-alpha}.

\subsection{Specifying the function}

Our approach is to specify the function~$f$ required by
Theorem~\ref{thm:main} via its Fourier transform. In this section
we discuss our choice for the Fourier transform of~$f$, give
explicit formulas for~$f$ in terms of its transform, and
show which constraints must be imposed on the transform so
that~$f$ is a real-valued,~$L^1$ and continuous function of positive
type which is~$\Sym(\Kcal)$-invariant.

Let~$N > 0$ be an integer and~$d \geq 1$ be an odd integer. Consider
the matrix-valued function~$\varphi$ given by
\begin{equation}
\label{eq:phi-def}
\varphi(a) = \bigl(\varphi_{r,s}(a)\bigr)_{r,s=-N}^N
= \biggl(\sum_{k=0}^d f_{r,s;k} a^{2k}\biggr)_{r,s=-N}^N.
\end{equation}
Notice that each~$\varphi(a)$ is a $(2N + 1) \times (2N + 1)$ matrix
whose entries are even univariate polynomials in the variable~$a$.

We define~$f$ as the function whose Fourier transform is
\begin{equation}
\label{eq:fhat-def}
\widehat{f}(a) = \varphi(a) e^{-\pi a^2}.
\end{equation}
Note that we express the operator~$\widehat{f}(a)$ in the
basis~$\chi_r$ for~$r \in \Z$, as discussed in
\S\ref{sec:harmonic}. Clearly, each~$\fh(a)$ is a trace-class operator. In fact, each~$\fh(a)$ has finite rank.

The reason for our choice for the Fourier transform is that it
makes it easy to compute the function~$f$. Let
\[
C_{r,s;k}(\rho) = \frac{\Gamma(k + 1 + |r - s| / 2)
  (\rho\sqrt{\pi})^{|r - s|}}{2 \pi^{k+1} \Gamma(|r-s| + 1)}.
\]
Then using (4.11.24) in Andrews, Askey, and Roy~\cite{AndrewsAR}, 
since~$J_n(z) = (-1)^n J_{-n}(z)$, we have
\[
\begin{split}
\int_0^\infty a^{2k+1} e^{-\pi a^2} J_{s - r}(2\pi a \rho)\, da
&= (-1)^{s - r} \int_0^\infty a^{2k+1} e^{-\pi a^2} J_{|r-s|}(2\pi
a\rho)\, da\\
&=(-1)^{s-r} C_{r,s;k}(\rho)\, {}_1 F_1\biggl({|r-s|/2 - k\atop |r-s|+1};
\pi\rho^2\biggr) e^{-\pi\rho^2},
\end{split}
\]
where~${}_1 F_1$ is the hypergeometric series.
Together with~\eqref{eq:f-formula} and~\eqref{eq:phi-def}, this
implies for~$f$ the formula
\begin{multline}
\label{eq:f-hypergeom}
f(\rho, \theta, \alpha) = 
\sum_{r,s=-N}^N \sum_{k=0}^d (-1)^{s-r} i^{s-r} e^{-i(s\alpha +
  (r-s)\theta)}
f_{r,s;k} C_{r,s;k}(\rho)\\
{}\cdot {}_1 F_1\biggl({|r-s|/2 - k\atop |r-s|+1};
\pi\rho^2\biggr) e^{-\pi\rho^2},
\end{multline}
where $(\rho, \theta, \alpha)$ parametrizes an element of~$\M(2)$ as
in \S\ref{sec:harmonic}.

It is immediately clear that, thanks to our choice of Fourier
transform,~$f$ is an~$L^1$ and continuous function; actually, it is
rapidly decreasing. So, by using Lemma~\ref{lem:pos-type}, we see that
if~$\fh(a)$ is a positive kernel for each~$a \geq 0$, then~$f$ is a
function of positive type. From the definition of~$\fh(a)$, we see
that~$\fh(a)$ is positive for every~$a\geq 0$ if and only if the
matrices~$\varphi(a)$ are positive semidefinite for every~$a$. Notice
that requiring~$\varphi(a)$ to be positive semidefinite includes
requiring~$\varphi(a)$ to be Hermitian. This on its turn we achieve by
imposing the constraint
\[
f_{r,s;k} = \overline{f_{s,r;k}}\quad
\text{for all~$r$, $s$, and~$k$.}
\]

We may further simplify~\eqref{eq:f-hypergeom} by imposing two extra
conditions on the coefficients~$f_{r,s;k}$. Namely, when~$r-s$ is even
and~$|r-s|/2 - k \leq 0$, the hypergeometric series
in~\eqref{eq:f-hypergeom} becomes a Laguerre polynomial; see also the treatment about the eigenfunction decomposition of the Hankel transform in the book \cite[Chapter 9]{Akhiezer} by Akhiezer. Indeed we
have (cf.~(6.2.2) in Andrews, Askey, and Roy~\cite{AndrewsAR})
\[
{}_1 F_1\biggl({|r-s|/2 - k\atop |r-s|+1}; \pi\rho^2\biggr)
= \frac{n!}{(|r-s|+1)_n} L_n^{|r-s|}(\pi\rho^2),
\]
where~$n = k - |r-s|/2$,
\[
(a)_n = a (a+1) \cdots (a + n - 1)\quad
\text{for $n > 0$ with $(a)_0 = 1$},
\]
and~$L_n^\alpha$ is the Laguerre polynomial of degree~$n$ and
parameter~$\alpha$.

So we impose on the coefficients~$f_{r,s;k}$ the constraints
\begin{equation}
\label{eq:small-k}
f_{r,s;k} = 0\quad\text{if~$r - s$ is odd or $k < |r-s|/2$}.
\end{equation}
Then~\eqref{eq:f-hypergeom} becomes
\begin{multline}
\label{eq:f-laguerre}
f(\rho, \theta, \alpha) = 
\sum_{{\scriptstyle r,s=-N\atop\scriptstyle r-s\ \rm even}}^N \sum_{k=|r-s|/2}^d (-1)^{|r-s|/2} e^{-i(s\alpha +
  (r-s)\theta)}
f_{r,s;k}\\
\cdot D_{r,s;k}(\rho) L_n^{|r-s|}(\pi\rho^2) e^{-\pi\rho^2},
\end{multline}
where~$D_{r,s;k}(\rho) = C_{r,s;k}(\rho) n! / (|r-s| + 1)_n$.

To ensure that~$f$ is a real-valued function, we observe
from~\eqref{eq:u-coeffs} that when~$r-s$ is even,~$u^a_{r,s}(\rho,
\theta, \alpha) = \overline{u^a_{-r,-s}(\rho, \theta, \alpha)}$. Then
from~\eqref{eq:f-formula} it is clear that if~$\varphi_{r,s}(a) =
\overline{\varphi_{-r,-s}(a)}$ for all~$a \geq 0$ and~$r$, $s$,
function~$f$ is real-valued. So to ensure that~$f$ is real-valued it
suffices to impose the constraint
\begin{equation}
\label{eq:real-constraint}
f_{r,s;k} = \overline{f_{-r,-s;k}}\quad
\text{for all~$r$, $s$, and~$k$}.
\end{equation}

Finally, we would like to impose constraints on the
coefficients~$f_{r, s;k}$ so as to make function~$f$
$\Sym(\Kcal)$-invariant, that is, so as to have
\[
f(\rho, \theta + l 2\pi / 5, \alpha) 
= f(\rho, \theta, \alpha)\quad
\text{for~$l = 0$, \dots,~$4$}.
\]

From~\eqref{eq:f-laguerre}, it is easy to see that one way of
achieving this is to require that
\[
f_{r, s; k} = 0\quad\text{whenever~$r - s \not\equiv 0\pmod{5}$}.
\]
Since we already set~$f_{r, s; k} = 0$ when~$r - s$ is odd, we end up
with the constraint
\begin{equation}
\label{eq:mod-10}
f_{r, s; k} = 0\quad\text{whenever~$r - s \not\equiv 0\pmod{10}$}.
\end{equation}

To finish, we summarize the constraints imposed on the
coefficients~$f_{r, s;k}$:

\begin{enumerate}
\item We consider only the pairs~$r$, $s$ such that~$r-s \equiv
  0\pmod{10}$ and we set~$f_{r, s; k} = 0$ if~$k < |r-s|/2$. This has
  a double effect: It simplifies the hypergeometric series into a
  Laguerre polynomial and makes the function $\Sym(\Kcal)$-invariant;

\item We set~$f_{r, s; k} = \overline{f_{s, r; k}}$ for all~$r$, $s$,
  and~$k$. This makes the matrices~$\varphi(a)$ Hermitian. We then
  require these matrices to be positive semidefinite; this ensures
  that function~$f$ is of positive type;

\item We set~$f_{r, s; k} = \overline{f_{-r, -s; k}}$ for all~$r$,
  $s$, and~$k$. This ensures that function~$f$ is real-valued.
\end{enumerate}

With these constraints, we obtain the following formula for~$f$:
\begin{multline}
\label{eq:f-final}
f(\rho, \theta, \alpha) = 
\sum_{{\scriptstyle r,s=-N\atop\scriptstyle r-s \equiv 0\ ({\rm mod}\ 10)}}^N \sum_{k=|r-s|/2}^d (-1)^{|r-s|/2} e^{-i(s\alpha +
  (r-s)\theta)}
f_{r,s;k}\\
\cdot D_{r,s;k}(\rho) L_n^{|r-s|}(\pi\rho^2) e^{-\pi\rho^2}.
\end{multline}

\subsection{A semidefinite programming formulation: basic setup}

Recall our goal is to describe a semidefinite programming problem
whose solutions correspond to functions~$f \in L^1(\M(n))$ satisfying
the conditions of Theorem~\ref{thm:main}. In this
section, we take a first step by showing how to formulate the problem
of finding a function~$\varphi$ like~\eqref{eq:phi-def} as a
semidefinite programming problem.

We start by making an extra assumption, namely that all coefficients~$f_{r, s; k}$ are real. This allows us to work exclusively with real matrices, making the semidefinite programming problem we obtain smaller. Then the optimization problem will be solvable by state-of-the-art semidefinite programming solvers which are numerically stable. In principle, however, everything we
describe can be extended to the more general setting of complex
coefficients. \textit{It could be}, though we do not know,
that such a restriction to real numbers greatly worsens the bound
that can be obtained via our approach.

So let~$\varphi$ be given as in~\eqref{eq:phi-def} with
\[
f_{s, r;k} =
f_{r, s; k} = f_{-r,-s;k} = f_{-s, -r; k}\quad\text{for all~$r$, $s$,
  and~$k$}
\]
as we
require. Write~$y = (y_{-N}, \ldots, y_N)$ and consider the polynomial
\[
\sigma(a, y) = \sum_{{\scriptstyle r, s = -N\atop\scriptstyle r-s
    \equiv 0\ ({\rm mod}\ 10)}}^N \sum_{k=|r-s|/2}^d f_{r,s;k} a^{2k} y_r
y_s.
\]
Then~$\varphi(a)$ is positive semidefinite for all~$a$ if and only
if~$\sigma$ is a sum of squares (see~\S\ref{sec:sdpsos}). (Here, it is
easy to see that if~$\sigma$ is a sum of squares, then~$\varphi(a)$ is
positive semidefinite for all~$a$. The converse is also true; for a
proof see Choi, Lam, and Reznick~\cite{ChoiLR}. This fact is related
to the Kalman-Yakubovich-Popov lemma in systems and control; see the
discussion in Aylward, Itani, and Parrilo~\cite{AylwardIP}.)

The constraint that~$\sigma$ is a sum of squares can on its turn be
formulated in terms of positive semidefinite matrices. Following the
recipe given on \S\ref{sec:sdpsos}, one would obtain a semidefinite
programming formulation in terms of a single variable matrix of large
size. In our case, however, since~$\sigma$ is an even polynomial
in~$a$ and since the product~$y_r y_s$ only appears when~$r-s \equiv 0
\pmod{10}$, we may block-diagonalize the variable matrix, obtaining a
formulation in terms of smaller matrices, as we show now.

To this end,
let~$P_0$, $P_1$, \dots\ be a sequence of real, even, univariate
polynomials such that~$P_k$ has degree~$2k$. For~$j = 0$, \dots,~$9$,
let
\[
\Ical_j = \{\, r \in \Z : -N \leq r \leq N\quad\text{and}\quad r \equiv
  j\pmod{10}\,\}.
\]
For~$i = 0$, $1$ and~$j = 0$, \dots,~$9$, consider the matrix~$V^{ij}$
with rows and columns indexed by~$\{0, \ldots, \floor{d/2}\} \times
\Ical_j$ such that
\[
V^{ij}_{(l,r)(l',s)} = a^{2i} P_l(a) P_{l'}(a) y_r y_s
\]
for all~$l$, $l' = 0$, \dots,~$\floor{d/2}$ and~$r$, $s \in
\Ical_j$. Notice the entries of~$V^{ij}$ are even polynomials in~$a$.

Then~$\sigma$ is a sum of squares if and only if there are real,
positive semidefinite matrices~$Q^{ij}$, of appropriate dimensions,
such that
\[
\sigma = \sum_{i=0}^1 \sum_{j=0}^9 \langle Q^{ij}, V^{ij}\rangle,
\]
where~$\langle A, B \rangle = \tr(B^* A)$ denotes the trace inner
product between matrices~$A$ and~$B$.

Here, it is also important to observe that the symmetry
constraints~$f_{r, s;k} = f_{s,r;k}$ are implied by the fact that the
matrices~$Q^{ij}$ are symmetric.

So finding real numbers~$f_{r,s;k}$ such that~$f_{r,s;k} = f_{s,r;k}$
and such that~$\varphi(a)$ is positive semidefinite for all~$a$
amounts to finding real positive semidefinite matrices~$Q^{ij}$. Also
the other constraints that we imposed on the coefficients~$f_{r,s;k}$
can be represented as linear constraints on the entries of
the~$Q^{ij}$ matrices, as we show now.

For~$r$, $s$, $k$ with~$r-s\equiv 0\pmod{10}$, let~$j \in
\{0,\ldots,9\}$ be such that $r$, $s \in \Ical_j$. For~$i = 0$, $1$,
consider the matrix~$F^i_{r,s;k}$ with rows and columns indexed
by~$\{0,\ldots,\floor{d/2}\} \times \Ical_j$ such that
\[
(F^i_{r,s;k})_{(l,r)(l',s)} = \coeff(a^{2k}, a^{2i} P_l(a) P_{l'}(a))
\]
for all~$l$, $l' = 0$, \dots,~$\floor{d/2}$, where for a given
polynomial~$p$, $\coeff(a^k, p)$ is the coefficient of monomial~$a^k$
in~$p$. Then we obtain the coefficients~$f_{r,s;k}$ from the
matrices~$Q^{ij}$ by the formula
\[
f_{r,s;k} = \sum_{i=0}^1 \langle F^i_{r,s;k}, Q^{ij}\rangle.
\]

So constraints~\eqref{eq:small-k} and~\eqref{eq:real-constraint}
become
\[
\begin{split}
&\sum_{i=0}^1 \langle F^i_{r,s;k}, Q^{ij}\rangle = 0\quad
\text{if~$k < |r-s|/2$},\\
&\sum_{i=0}^1 (\langle F^i_{r,s;k}, Q^{ij}\rangle - \langle
F^i_{-r,-s;k}, Q^{ij'}\rangle) = 0\quad
\text{for all~$r$, $s$, and~$k$},
\end{split}
\]
where~$r$, $s \equiv j\pmod{10}$ and~$-r, -s\equiv
j'\pmod{10}$. Notice that constraint~\eqref{eq:mod-10} is already
implicit in our formulation, because we enforce by construction that
only pairs~$r$, $s$ with~$r-s\equiv 0\pmod{10}$ occur.

Also the function~$f$ can be computed from matrices~$Q^{ij}$. To see
how, for~$r$, $s = -N$, \dots,~$N$ such that~$r-s\equiv 0\pmod{10}$
and for~$k\geq |r-s|/2$, set
\begin{equation}
\label{eq:tau-def}
[\tau_{r,s}(a^{2k})](\rho, \theta, \alpha)
= (-1)^{|r-s|/2} e^{-i(s\alpha + (r - s)\theta)} D_{r,s;k}(\rho)
L_n^{|r-s|}(\pi\rho^2),
\end{equation}
where~$n = k - |r-s|/2$. When~$k < |r-s|/2$, we
set~$\tau_{r,s}(a^{2k}) = 0$, and then we extend~$\tau_{r,s}$ linearly
to all even polynomials in the variable~$a$.

For~$i = 0$, $1$ and~$j = 0$, \dots,~$9$, consider the
matrix~$\Fcal^{ij}$ with rows and columns indexed by~$\{0, \ldots,
\floor{d/2}\} \times \Ical_j$ such that
\[
[\Fcal^{ij}(\rho, \theta, \alpha)]_{(l,r)(l',s)} = \tau_{r,s}(a^{2i}
P_l(a) P_{l'}(a))
\]
for all~$l$, $l' = 0$, \dots,~$\floor{d/2}$ and~$r$, $s \in
\Ical_j$. Then, in view of~\eqref{eq:f-final} and since~$f_{r,s;k} =
0$ whenever~$k < |r-s|/2$, we have
\begin{equation}
\label{eq:f-matrices}
f(\rho, \theta, \alpha) = \sum_{i=0}^1 \sum_{j=0}^9
\langle \Fcal^{ij}(\rho, \theta, \alpha), Q^{ij}\rangle
e^{-\pi\rho^2}.
\end{equation}

\subsection{Ensuring nonpositiveness}

How can we ensure that function~$f$, given by~\eqref{eq:f-matrices},
satisfies constraint~(ii) of Theorem~\ref{thm:main}? This we do also in
terms of semidefinite programming constraints.

First, observe that we require~$f(x, A) \leq 0$ whenever~$\Kcal^\circ
\cap (x + A\Kcal^\circ) = \emptyset$. The latter happens if and only
if~$x \notin (\Kcal - A\Kcal)^\circ$, where~$\Kcal - A\Kcal$ is the
\defi{Minkowski difference} of~$\Kcal$ and~$A\Kcal$:
\[
\Kcal - A\Kcal = \{\, y-z : y\in\Kcal,\ z\in A\Kcal\,\}.
\]

The Minkowski difference~$\Kcal - A\Kcal$ is a polygon for all~$A \in
\so(2)$. Its vertices can be explicitly determined;
Figure~\ref{fig:mink} shows the Minkowski difference when~$A =
A(\alpha)$ (as defined in~\eqref{eq:A-alpha}) for~$\alpha \in
[-2\pi/10, 2\pi/10]$. By the symmetry of~$\Kcal$, this gives a full
characterization of the shape of the Minkowski difference for
all~$\alpha$.

\begin{figure}[tb]
\setbox0=\hbox{\includegraphics{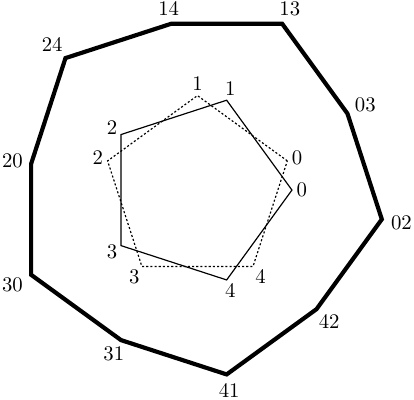}}
\def\picbox#1{\hbox to\wd0{\hfil\includegraphics{#1}\hfil}}

\noindent
\centerline{\hbox to0pt{\hss\picbox{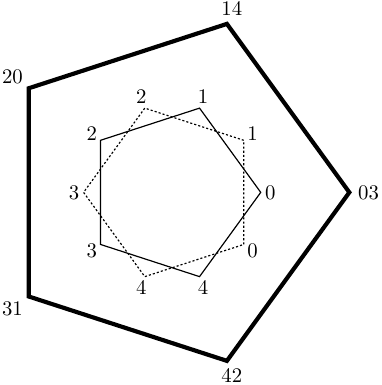}\hskip5mm\picbox{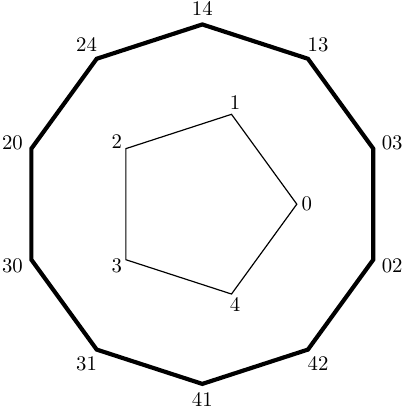}\hss}}

\vskip5mm

\noindent
\centerline{\hbox
  to0pt{\hss\picbox{m2.pdf}\hskip5mm\picbox{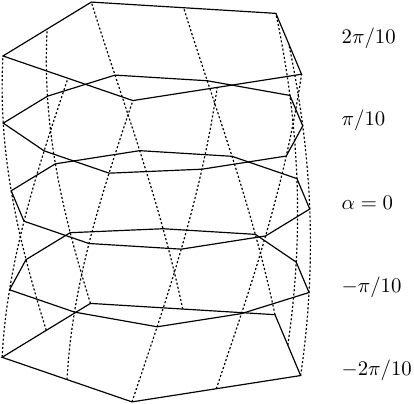}\hss}}

\vskip5mm

\caption{From left to right, top to bottom. In the first three
  pictures, we see the Minkowski difference $\Kcal - A(\alpha)\Kcal$
  (the outer shape) for~$\alpha = -2\pi/10$, $0$, and~$\pi/10$. The
  dashed pentagon in the center corresponds to~$A(\alpha)\Kcal$. The
  vertices of the pentagons are numbered from~$0$ to~$4$. The vertices
  of the Minkowski difference are numbered~$ij$, meaning that they
  correspond to~$x - y$, where~$x$ is the $i$th vertex of~$\Kcal$
  and~$y$ is the $j$th vertex of~$A(\alpha)\Kcal$. In the last picture
  we show the three-dimensional set~$\{\, (x, \alpha) : x \in \Kcal -
  A(\alpha)\Kcal\,\}$. Here,~$\alpha$ is on the vertical axis; every
  section perpendicular to the vertical axis corresponds to a
  Minkowski difference~$\Kcal - A(\alpha)\Kcal$.}
\label{fig:mink}
\end{figure}

Our approach to ensure that~$f$ is nonpositive outside of~$(\Kcal -
A\Kcal)^\circ$ consists of two steps. First, we observe that all
vertices of~$\Kcal - A\Kcal$ have norm at most~$1$. This implies that
we must have~$f(x, A) \leq 0$ whenever~$\|x\|\geq 1$. This condition
on~$f$ can be expressed in terms of sums of squares constraints.

Indeed, by writing~$z_1 = e^{i\theta}$ and~$z_2 = e^{i(\alpha -
  \theta)}$, we may rewrite~\eqref{eq:tau-def} as
\[
[\tau_{r,s}(a^{2k})](\rho, z_1, z_2) = 
(-1)^{|r-s|/2} z_1^{-r} z_2^{-s} D_{r,s;k}(\rho)
L_n^{|r-s|}(\pi\rho^2).
\]
In view of~\eqref{eq:f-matrices}, if we then have
\[
\sum_{i=0}^1 \sum_{j=0}^9 \langle \Fcal^{ij}(\rho, z_1, z_2),
Q^{ij}\rangle \leq 0\quad
\text{for all~$\rho \geq 1$},
\]
we have~$f(\rho, \theta, \alpha) \leq 0$ whenever~$\rho \geq 1$, as we
want.

For~$j = 0$, \dots,~$9$, consider the set
\[
\Pcal_j = \{\, (r, s) : 0 \leq r, s \leq N\quad\text{and}\quad
r-s\equiv j\pmod{10}\,\}.
\]
For~$i = 0$, $1$, $j = 0$, \dots,~$9$, consider the matrix~$W^{ij}$
with rows and columns indexed
by~$\{0, \ldots,\floor{d/2}\} \times \Pcal_j$ such that
\[
W^{ij}_{(l,p)(l',p')}(\rho, z_1, z_2) = (\rho^i P_l(\rho) z_1^{-u}
z_2^{-v}) (\rho^i P_{l'}(\rho) z_1^{u'} z_2^{v'}),
\]
where~$p = (u, v)$ and~$p' = (u', v')$  with~$p$, $p' \in \Pcal_j$,
and~$l$, $l' = 0$, \dots,~$\floor{d/2}$.

If there are real positive semidefinite matrices~$R^{ij}$ for~$i = 0$,
$1$ and~$j = 0$, \dots,~$9$, and~$S^j$ for~$j = 0$, \dots,~$9$, such
that
\begin{multline}
\label{eq:cylinder-sos}
\sum_{i=0}^1 \sum_{j=0}^9 (\langle\Fcal^{ij}(\rho, z_1, z_2),
Q^{ij}\rangle + \langle W^{ij}(\rho, z_1, z_2), R^{ij}\rangle)\\
{}+\sum_{j=0}^9 \langle (\rho^2 - 1) W^{0j}(\rho, z_1, z_2),
S^j\rangle = 0,
\end{multline}
then~$f(\rho, \theta, \alpha) \leq 0$ for all~$\rho \geq
1$. Notice~\eqref{eq:cylinder-sos} is a polynomial identity on
variables~$\rho$, $z_1$, $z_1^{-1}$, $z_2$, and~$z_2^{-1}$. In other
words, the left-hand side defines a polynomial and the identity above
states that this polynomial must be identically zero. To see
that~\eqref{eq:cylinder-sos} implies that~$f(\rho, \theta, \alpha)
\leq 0$ whenever~$\rho \geq 1$, one only has to notice that, for~$\rho
\geq 1$ and~$\theta, \alpha \in [0, 2\pi]$, the Hermitian matrices
\[
W^{ij}(\rho, e^{i\theta}, e^{i(\alpha-\theta)})\quad\text{and}\quad
(\rho^2 - 1)W^{0j}(\rho, e^{i\theta}, e^{i(\alpha-\theta)})
\]
are positive semidefinite, and then all inner products
in~\eqref{eq:cylinder-sos} become nonnegative.

Constraint~\eqref{eq:cylinder-sos} is not enough to ensure, however,
that~$f$ is nonpositive outside of the Minkowski difference. To ensure
nonpositiveness in the remaining region, we use a discretization
heuristic: We pick a sample of triples~$(\rho, \theta, \alpha)$
with~$\rho \leq 1$ for which we have to ensure that~$f(\rho, \theta,
\alpha) \leq 0$ and we do so explicitly for every point of the sample
using~\eqref{eq:f-matrices}. Afterwards, we have to analyze the
solution obtained in order to check that it indeed satisfies
condition~(ii) of Theorem~\ref{thm:main}. We will give details on this
approach in the next section.

One may model the constraint that~$f$ is nonpositive outside the
Minkowski difference using only sums of squares, without using the
discretization approach. The sizes of the matrices get very large,
however, making this approach computationally infeasible.

\subsection{The semidefinite programming problem and how to solve it}

We now describe the semidefinite programming problem we solve to
obtain upper bounds for the pentagon packing density.

Let~$N > 0$ be an integer and~$d \geq 1$ be an odd
integer. Let~$\Scal$ be a finite set of triples~$(\rho, \theta,
\alpha)$ with~$\rho \leq 1$ corresponding to elements~$(x, A) \in
\M(2)$ such that~$\Kcal^\circ \cap (x + A\Kcal^\circ) = \emptyset$. We
consider the following semidefinite programming problem:
\medbreak

\noindent
{\bf Problem A.}\enspace Find real, positive semidefinite
matrices~$Q^{ij}$, $R^{ij}$ for~$i = 0, 1$ and~$j = 0$, \dots,~$9$,
and~$S^j$ for~$j = 0$, \dots,~$9$, that minimize
\[
\sum_{i=0}^1\sum_{j=0}^9 \langle\Fcal^{ij}(0,0,0), Q^{ij}\rangle
\]
subject to the constraints
\begin{align}
&\label{eq:f-coeff-zero}\sum_{i=0}^1 \langle F^i_{r,s;k}, Q^{ij}\rangle = 0\quad
\text{if $k < |r-s|/2$, where $r$, $s\equiv j\pmod{10}$,}\\
&\label{eq:f-coeff-real}\sum_{i=0}^1 (\langle F^i_{r,s;k}, Q^{ij}\rangle - \langle
F^i_{-r,-s;k}, Q^{ij'}\rangle) = 0\quad
\text{where $r$, $s\equiv j\pmod{10}$}\\[-7mm]
&\null\phantom{\sum_{i=0}^1 \langle F^i_{r,s;k}, Q^{ij}\rangle - \langle
F^i_{-r,-s;k}, Q^{ij'}\rangle = 0}
\quad\ \text{and~$-r$, $-s\equiv j'\pmod{10}$,}\nonumber\\
&\label{eq:non-pos}\sum_{i=0}^1 \sum_{j=0}^9 (\langle \Fcal^{ij}(\rho, z_1, z_2),
Q^{ij}\rangle + \langle W^{ij}(\rho, z_1, z_2), R^{ij}\rangle)\\[-4mm]
&\hskip5cm\null + \sum_{j=0}^9 \langle (\rho^2 - 1) W^{0j}(\rho, z_1,
z_2), S^j\rangle = 0,\nonumber\\
&\label{eq:sample}\sum_{i=0}^1 \sum_{j=0}^9 \langle \Fcal^{ij}(\rho, \theta, \alpha),
Q^{ij}\rangle \leq 0\quad\text{for all $(\rho, \theta, \alpha) \in
  \Scal$},\\
&\label{eq:normal}\sum_{i=0}^1 \langle F^i_{0,0;0}, Q^{i0}\rangle = 1.
\end{align}
\medbreak

Conditions \eqref{eq:f-coeff-zero}--\eqref{eq:sample} were already
discussed in the previous sections.  Notice this is indeed a
semidefinite programming problem. In fact, the objective function and
all constraints but~\eqref{eq:non-pos} are clearly linear. As for the
polynomial identity~\eqref{eq:non-pos}, one only has to observe that
it can be turned into linear constraints by using the fact that a
polynomial is identically zero if and only if each monomial has a zero
coefficient (cf.~\S\ref{sec:sdpsos}).

Of Problem~A we have to explain our choice of objective function and
also the meaning of constraint~\eqref{eq:normal}. To obtain the best
possible bound from Theorem~\ref{thm:main}, we wish to minimize~$f(0,
I) / \lambda$, where
\[
\lambda = \int_{\M(2)} f(x, A)\, d(x, A).
\]
Constraint~\eqref{eq:normal} is a normalization constraint,
setting~$\lambda = 1$. Indeed, one has~$\lambda = f_{0,0;0}$, since
from the definition of~$\widehat{f}$ and the inversion formula
(cf.~\S\ref{sec:harmonic}) we have
\[
\begin{split}
f_{0,0;0} = (\widehat{f}(0))_{0,0} &=\langle \widehat{f}(0) \one,
\one\rangle\\
&=\biggl\langle \int_{\M(2)} f(x, A) U^0_{(x, A)^{-1}} \one\, d(x, A),
\one\biggr\rangle\\
&= \lambda \langle \one, \one \rangle\\
&= \lambda,
\end{split}
\]
where~$\one \in L^2(S^1)$ is the constant one function, so
that~$U^0_{(x, A)^{-1}} \one = \one$. Now, the objective function
evaluates~$f(0, I)$, that we wish to minimize.

To be able to solve Problem~A on the computer, the choice of the
sequence~$P_0$, $P_1$, \dots\ of polynomials which we use to define
our matrices is essential. A bad choice here can lead to numerical
instability that might prevent us from solving the problem.

In particular, we have observed that the monomial basis performs
specially badly. A much better choice are normalized Laguerre
polynomials, as had been observed in a similar setting by de Laat,
Oliveira, and Vallentin~\cite{LaatOV}. Namely, we set
\[
P_k(x) = \mu_k^{-1} L_k^0(2\pi x^2),
\]
where~$\mu_k$ is the absolute value of the coefficient of~$L_l^0(2\pi
x^2)$ with largest absolute value.

Also essential to the stability of Problem~A is the choice of the
basis used to express polynomial identity~\eqref{eq:non-pos}. Again,
the monomial basis is a poor choice. Instead we use the basis
\[
P_k(\rho^2) z_1^{-r} z_2^{-s}
\]
for~$k = 0$, \dots,~$d$ and $-N \leq r, s \leq N$ such that $r-s
  \equiv 0\pmod{10}$. 

This means that in order to express constraint~\eqref{eq:non-pos}, we
expand the corresponding polynomial in the above basis, and then
require each coefficient of the expansion to be zero.

In preliminary tests with reasonably dense samples for
constraint~\eqref{eq:sample}, we observed that most variables in
Problem~A did not seem to play a role, at least for the values of~$d$
and~$N$ that we considered. So we decided to discard all variable matrices
except for
\[
\text{$Q^{00}$, $Q^{05}$, $Q^{10}$, $Q^{15}$, $R^{00}$, $R^{05}$,
  $S^0$, and $S^5$,}
\]
and we observed that this did not have much effect on the optimal
value of the problem, while providing for simpler and more stable
problems. From now on, when we refer to Problem~A it should be
understood that we only use the variables listed above.

We now have a complete description of the semidefinite programming
problem to be solved, let us sketch how we obtained the bound of~$0.98103$
for the pentagon packing density.

We first solve Problem~A (with less variables, as explained above)
for~$d = 11$ and~$N = 5$, using a sample with~537 points. This sample
we pick as follows. We first pick~$5$ uniformly spaced values
for~$\alpha$ in~$[-2\pi/10, 0]$, starting with~$-2\pi/10$ and ending
with~$0$. For each such value of~$\alpha$, we pick in the square~$[-1,
  1]^2$ a uniformly spaced grid of~$50 \times 50$ points, and add to
the sample all triples~$(\rho, \theta, \alpha)$, where~$(\rho,
\theta)$ corresponds to a grid point outside of the Minkowski
difference~$(\Kcal - A(\alpha)\Kcal)^\circ$ and such that~$\rho \leq
1$. Moreover, the symmetry of~$\Kcal$ allows us to restrict our sample
considerably --- Figure~\ref{fig:sample} has an example.

\begin{figure}[t]
\centerline{\includegraphics{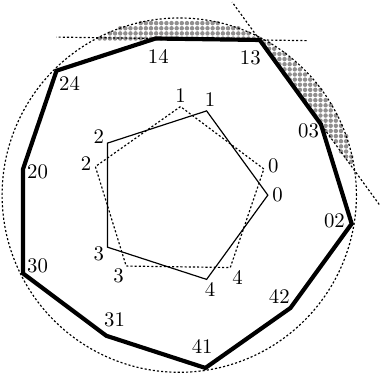}}
\vskip5mm

\caption{The points in gray are an example of a sample used in
  Problem~A; here we show the points in the sample for~$\alpha =
  \pi/10$. Each facet~$F$ of the Minkowski difference defines a
  line~$l_F$, its supporting hyperplane, and for the sample we would
  then pick all points in the grid that are inside the circle of
  radius~$1$ and that lie, for some facet~$F$ of the Minkowski
  difference, on the side of~$l_F$ that does not contain the
  origin. Since we work with $\Sym(\Kcal)$-invariant functions,
  however, we need not choose all these points: It suffices to
  consider only two adjacent facets of the Minkowski difference,
  instead of all the facets.}
\label{fig:sample}
\end{figure}

We observed, by evaluating the function~$f$ obtained via this
approach, that this small sample is already enough to enforce
condition~(ii) of Theorem~\ref{thm:main} on most of the required
domain. To really obtain a function~$f$ satisfying the conditions of
Theorem~\ref{thm:main}, however, we have to work a bit more.

Since we use a numerical solver for semidefinite programming, the
solutions we obtain for Problem~A are not really feasible, but almost
feasible. So we cannot be \textit{a priori} certain that the bound
given by Problem~A is really an upper bound.

To deal with this issue, we use the same approach outlined by de Laat,
Oliveira, and Vallentin~\cite{LaatOV}, which we briefly explain
here. First, we solve Problem~A in order to get an estimate of its
optimal value; say~$z^*$ is the numerical optimal value obtained. Then, we solve
a version of Problem~A in which the objective function is removed but
a constraint
\[
\sum_{i=0}^1\sum_{j=0}^9 \langle\Fcal^{ij}(0,0,0), Q^{ij}\rangle \leq
z^* + 10^{-5}
\]
is added. 

This problem is a feasibility problem, and for this reason the solver
will return a solution that is strictly feasible, i.e., a solution in
which the solution matrices are \textit{positive definite}, if one can
be found.

In this way, we manage to obtain a solution of Problem~A having
objective value close to what the optimal value is supposed to be, in
which each matrix has a minimum eigenvalue around~$10^{-6}$, whereas
the constraints are satisfied up to an absolute error
of~$10^{-9}$. By projecting the solution obtained onto the affine
subspace generated by constraints~\eqref{eq:f-coeff-zero},
\eqref{eq:f-coeff-real}, \eqref{eq:non-pos}, and~\eqref{eq:normal},
using double-precision floating-point arithmetic, we manage to drop
the absolute error to~$10^{-22}$, while not changing much the minimum
eigenvalues of the solution matrices. We give the matrices $Q^{00},
Q^{10}, Q^{05}, Q^{15}$ parametrizing function $f$ in Figure~\ref{fig:function}.

\begin{figure}
\includegraphics[scale=0.85]{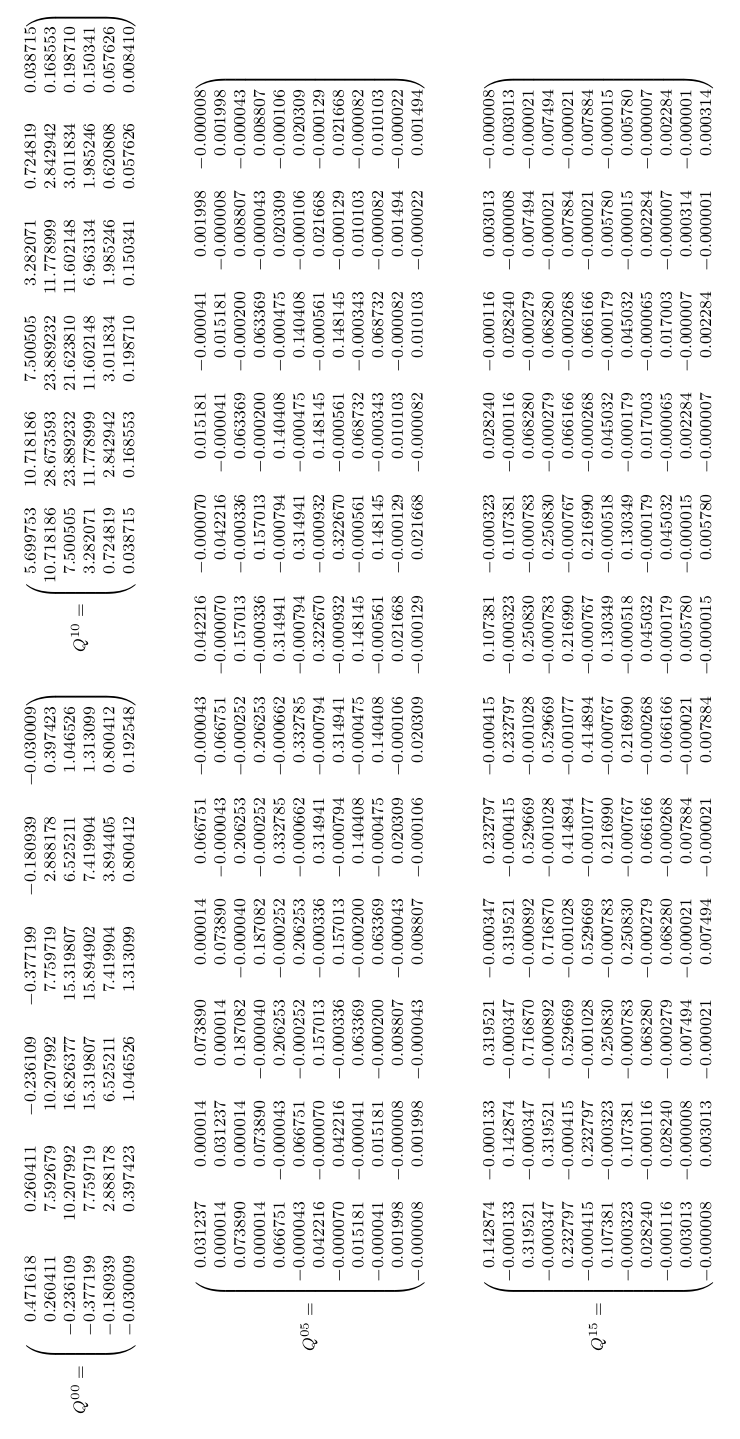}
\caption{Matrices $Q^{00}, Q^{10}, Q^{05}, Q^{15}$ parametrizing the
  function $f$ after the projection.}
\label{fig:function}
\end{figure}

So the approach detailed by de Laat, Oliveira, and
Vallentin~\cite{LaatOV} applies. Namely, since the minimum eigenvalues
of the solution matrices are big compared to the absolute errors, we
may be sure that by changing the solution matrices slightly, we may
ensure that the constraints are satisfied, thus obtaining a truly
feasible solution, without significantly changing the objective
value. Notice that we do not need to carry out this change in
practice, it suffices to know that it can be done.

Finally, we still have to show that the function~$f$ thus obtained
satisfies condition~(ii) of Theorem~\ref{thm:main}. We have said
that~$f$ satisfies condition~(ii) for most of the points on the
required domain. For instance, since we have
constraint~\eqref{eq:non-pos}, we know that~$f(\rho, \theta, \alpha)
\leq 0$ for all~$\rho \geq 1$. There are, however, points~$(\rho,
\theta, \alpha)$ with~$\rho \leq 1$ for which we have~$f(\rho, \theta,
\alpha)$ positive, while~(ii) would require this value to be
nonpositive.

Though~$f$ does not satisfy condition~(ii) of Theorem~\ref{thm:main}
for the pentagon~$\Kcal$, it satisfies this condition once we
enlarge~$\Kcal$ slightly. Indeed,~$f$ satisfies condition~(ii) for the
pentagon~$1.02\Kcal$. This we may verify by picking a fine enough
sample of points~$(\rho, \theta, \alpha)$ with~$\rho \leq 1$ for
which~$f$ has to be nonpositive, and computing the minimum value
of~$f$ on this sample using 256-bit-precision floating-point.
By computing the derivatives of~$f$, we may estimate how fine the
sample has to be and how large the absolute value of the minimum
of~$f$ on the sample has to be, in order for us to be sure that~$f$ is
nonpositive in the whole required region.

A side effect of our restriction of the variables is that the
function~$f$ we obtain is by construction such that
\[
f(\rho, \theta, \alpha + l 2\pi / 5) = f(\rho, \theta, \alpha)
\]
for all integer~$l$ (cf.~\eqref{eq:f-final}). This and the symmetry
of~$\Kcal$ helps us restrict the sample to points~$(\rho, \theta,
\alpha)$ with~$\alpha \in [-2\pi/10, 2\pi/10]$. To obtain our bound,
we had to use a sample of about 6.5~million points to check that~$f$
satisfies condition~(ii) of Theorem~\ref{thm:main}. Details of this
procedure can be found in the paper \cite{DostertGOV} by Dostert, Guzm\'an,
Oliveira, and Vallentin.

Enlarging the body~$\Kcal$ worsens the bound given by
Theorem~\ref{thm:main}, but since we consider a small enlargement
of~$\Kcal$, we still manage to obtain the bound of~$0.98103$.

Finally, we mention some of the computational tools used to generate
the semidefinite programming problem and solve it. To generate the
problem, we use a C++ program with a custom-made C++ library for
generating semidefinite programming problems, in particular dealing
with sums of squares constraints. As a solver we used
CSDP~\cite{CSDP}, and to analyze the resulting solution and check that
it is feasible we used a mix of SAGE~\cite{SAGE} and C++.

\section*{Acknowledgements}

We are thankful to Pier Daniele Napolitani and Claudia Addabbo from
the Maurolico Project\footnote{\url{http://maurolico.free.fr}}, who
provided us with a transcript of Maurolico's manuscript. In
particular, Claudia Addabbo provided us with a draft of her commented
Italian translation of the manuscript.

\end{document}